\documentclass[10pt]{amsart}

\usepackage{amsmath, amsthm, amssymb, graphicx}

\newcounter{citedtheorems}

\newtheorem{defn}{Definition}[section]
\newtheorem{theorem}[defn]{Theorem}
\newtheorem*{theorem-n}{Main Theorem}
\newtheorem*{thm-n}{Theorem}
\newtheorem*{claim-n}{Claim}
\newtheorem{thm-lit}[citedtheorems]{Theorem}
\newtheorem{defn-lit}[citedtheorems]{Definition}
\newtheorem{fact}[defn]{Fact}
\newtheorem{cor}[defn]{Corollary}

\newtheorem{concl}[defn]{Conclusion}
\newtheorem{conv}[defn]{Convention}
\newtheorem{claim}[defn]{Claim}

\newtheorem{lemma}[defn]{Lemma}
\newtheorem{obs}[defn]{Observation}
\newtheorem{rmk}[defn]{Remark}

\newtheorem{disc}[defn]{Discussion}

\newtheorem{o-conj}[defn]{Original Conjecture}

\newcommand{\lost}{\L os' }

\newcommand{\br}{\vspace{2mm}}

\newcommand{\kleq}{\trianglelefteq}

\newcommand{\tlf}{\trianglelefteq}

\newcommand{\step}{\vspace{3mm}\noindent\emph}   
\newcommand{\lp}{\emph{(}}
\newcommand{\rp}{\emph{)}}
\newcommand{\mpp}{\mathbf{P}}
\newcommand{\mch}{\mathcal{H}}

\newcommand{\de}{\mathcal{D}}
\newcommand{\ee}{\mathcal{E}}
\newcommand{\eff}{\mathcal{F}}

\newcommand{\fss}{{\mathcal{P}}_{\aleph_0}}

\newcommand{\trv}{\textbf{t}} 
\newcommand{\uu}{\mathcal{U}}

\newcommand{\gee}{\mathcal{G}}
\newcommand{\trg}{{T_{\mathbf{rg}}}}
\newcommand{\rn}{\operatorname{Range}}

\newcommand{\ba}{\mathfrak{B}}

\newcommand{\fin}{\operatorname{FI}}

\newcommand{\mk}{\mathcal{K}}
\newcommand{\el}{\emph{(}}
\newcommand{\er}{\emph{)}}
\newcommand{\rstr}{\upharpoonright}

\newcommand{\vp}{\varphi}
\newcommand{\lcf}{\operatorname{lcf}}
\newcommand{\cf}{\operatorname{cf}}

\newcommand{\ci}{\textbf{f}_i}
\newcommand{\bx}{\mathbf{x}}
\newcommand{\bz}{\mathcal{R}}
\newcommand{\mx}{\mathbf{x}}
\newcommand{\mh}{\mathbf{h}}

\newcommand{\fil}{\operatorname{fil}}
\newcommand{\mcp}{\mathcal{P}}
\newcommand{\ma}{\mathcal{A}}
\newcommand{\maa}{\mathbf{a}}
\newcommand{\mb}{\mathbf{b}}
\newcommand{\dom}{\operatorname{dom}}
\newcommand{\Hom}{\operatorname{Hom}}

\title[Model theoretic properties of ultrafilters...]{Model-theoretic properties of ultrafilters \\ built by independent families of functions}
\author{M. Malliaris and S. Shelah}\thanks{\emph{Thanks}: 
Malliaris was partially supported by NSF grant DMS-1001666 and by a G\"odel fellowship. 
Shelah was partially supported by the Israel Science Foundation grants 710/07 and 1053/11. 
This is paper 997 in Shelah's list.}

\address{Department of Mathematics, University of Chicago, 5734 S. University Avenue, Chicago, IL 60637, USA and
Einstein Institute of Mathematics, Edmond J. Safra Campus, Givat Ram, The Hebrew
University of Jerusalem, Jerusalem, 91904, Israel}
\email{mem@math.uchicago.edu}

\address{Einstein Institute of Mathematics, Edmond J. Safra Campus, Givat Ram, The Hebrew
University of Jerusalem, Jerusalem, 91904, Israel, and Department of Mathematics,
Hill Center - Busch Campus, Rutgers, The State University of New Jersey, 110
Frelinghuysen Road, Piscataway, NJ 08854-8019 USA}
\email{shelah@math.huji.ac.il}
\urladdr{http://shelah.logic.at}

\begin{document}



\begin{abstract}
Our results in this paper increase the model-theoretic precision
of a widely used method for building ultrafilters, and so advance 
the general problem of constructing ultrafilters whose ultrapowers have a precise degree of saturation.
We begin by showing that any flexible regular ultrafilter makes the product of an unbounded sequence
of finite cardinals large, {thus} saturating any stable theory. 
We then prove directly that a ``bottleneck'' in the inductive construction of a
regular ultrafilter on $\lambda$ (i.e. a point after which all antichains of $\mcp(\lambda)/\de$ have
cardinality less than $\lambda$) essentially prevents any subsequent ultrafilter from being flexible,
{thus} from saturating any non-low theory. 
The constructions are as follows.
First, we construct a regular filter $\de$ on $\lambda$ so that any ultrafilter extending $\de$ fails to 
$\lambda^+$-saturate ultrapowers of the random graph, {thus} of any unstable theory. 
The proof constructs the omitted random graph type directly.
Second, assuming existence of a measurable cardinal $\kappa$, we construct a regular ultrafilter on $\lambda > \kappa$ which is $\lambda$-flexible
but not $\kappa^{++}$-good, improving our previous answer to a question raised in Dow 1975.
Third, assuming a weakly compact cardinal $\kappa$, we construct an ultrafilter to show that $\lcf(\aleph_0)$ may be small while all symmetric cuts of cofinality $\kappa$ are realized. Thus certain families of pre-cuts may be realized while still failing to saturate any unstable theory.  
\end{abstract}

\maketitle

\section{Introduction}

Our work in this paper is framed by the longstanding open problem of Keisler's order, introduced in Keisler 1967 \cite{keisler} and
defined in \ref{keisler-order} below. 
Roughly speaking, this order allows one to compare the complexity of theories in terms of the relative 
difficulty of producing saturated regular ultrapowers. An obstacle to progress on this order has been the difficulty  
of building ultrafilters which produce a precise degree of saturation. 

Recent work of the authors (Malliaris \cite{mm-thesis}-\cite{mm5}, Malliaris and Shelah \cite{MiSh:996}-\cite{MiSh:998})
has substantially advanced our understanding of the interaction of ultrafilters and theories. 
Building on this work, in the current paper
and its companion \cite{MiSh:996} we address the problem of building ultrafilters with specific amounts of saturation. 
\cite{MiSh:996} focused on constructions of ultrafilters by products of regular and complete ultrafilters, 
and here we use the method of independent families of functions.

First used by Kunen in his 1972 ZFC proof of the existence of good ultrafilters, 
the method of independent families of functions
has become fundamental for constructing regular ultrafilters. 
The proofs in this paper leverage various inherent constraints of this method
to build filters with specified boolean combinations of model-theoretically meaningful properties, i.e. 
properties which guarantee or prevent realization of types. 

Our main results are as follows. Statements and consequences are given in more detail in \S \ref{s:description} below.
We prove that any ultrafilter $\de$ which is $\lambda$-flexible (thus: $\lambda$-o.k.) must have $\mu(\de) = 2^\lambda$, 
where $\mu(\de)$ is the  minimum size of a product of an unbounded sequence of natural numbers modulo $\de$ . 
Thus, a fortiori, $\de$ will saturate any stable theory. 
We prove that if, at any point in a construction by independent functions the cardinality of the range of 
the remaining independent family is strictly smaller than the index set, then essentially no subsequent ultrafilter can be flexible.
We then give our three main constructions. 
First, we show how to construct a filter so that no subsequent ultrafilter will saturate the random graph, \emph{thus}
no subsequent ultrafilter will saturate any unstable theory (see \ref{conventions} for this use of the word ``saturate'').
The proof explicitly builds an omitted type into the construction. 
Second, assuming the existence of a measurable cardinal $\kappa$, 
we prove that on any $\lambda \geq \kappa^+$ there is a regular ultrafilter
which is flexible but not good. This result improves our prior answer, in \cite{MiSh:996}, to a question from Dow 1975 \cite{dow} and
introduces a perspective which proved significant for \cite{MiSh:999}.
Third, we construct an example proving an a priori surprising nonimplication between realization of symmetric cuts and 
$\lcf(\aleph_0, \de)$, i.e. the coinitiality of $\omega$ in $(\omega, <)^\lambda/\de$.
That is, assuming the existence of a weakly compact cardinal $\kappa$, we prove that
for $\aleph_0 < \theta = \cf(\theta) < \kappa \leq \lambda$ there is a regular ultrafilter $\de$ on $\lambda$ such that
$\lcf(\aleph_0, \de) = \theta$ but $(\mathbb{N}, <)^\lambda/\de$ has no $(\kappa, \kappa)$-cuts. This appears counter to
model-theoretic intuition, since it shows some families of cuts in linear order can be realized without saturating any 
unstable theory. The proof relies on building long indiscernible sequences in the quotient Boolean algebra. 

For the model theoretic reader, we attempt to give a relatively self contained account 
of independent families and construction of ultrafilters as used here. We 
define all relevant properties of ultrafilters, many of
which correspond naturally to realizing certain kinds of types.
For the reader interested primarily in combinatorial set theory, note that while the model-theoretic point of view is fundamental, 
we deal primarily with ultrapowers of the random graph and of linear order; the arguments mainly require
familiarity with saturation, the random graph, and, ideally, the definitions of unstable, finite cover property, 
order property, independence property and strict order property. [Definitions 
like ``non-simple'' and ``non-low'' may be taken as black boxes in theorems about properties of filters.]  

Familiarity with Keisler's order beyond what is described below is not necessary for reading the present paper. For the interested reader, however,
\cite{MiSh:996} Sections 1-4 are a lengthy expository introduction to Keisler's order and what is known.

The paper is organized as follows.
Our methods, results and some consequences are described in \S \ref{s:description} below.
Section \ref{s:background} gives the key definitions and the necessary background 
on constructing ultrafilters via independent families of functions. 
Sections \S \ref{s:flex-mu}-\S \ref{s:weakly-compact} contain the proofs.

Thanks to Shimon Garti for some helpful comments, and to Simon Thomas for organizational remarks on 
this paper and its companion \cite{MiSh:996}. 

\setcounter{tocdepth}{1}
\tableofcontents \label{toc}

\section{Description of results} \label{s:description}

This section presents the results of the paper in more detail, under italicized headers, along with some consequences. 
We informally say that a (regular) ultrafilter $\de$ on $I$ saturates a theory $T$ to mean that whenever $M \models T$,
$M^I/\de$ is $|I|^+$-saturated. This phrasing is justified by Theorem \ref{backandforth}, \S \ref{s:background} below.
``Minimum,'' ``maximum'' refer to Keisler's order 
$\tlf$, \ref{keisler-order} below.
Theorem \ref{formula-corr} in \S \ref{s:review} may be a useful glossary for this introduction.

\br
\step{Flexible filters saturate stable theories.}
Flexibility was introduced in Malliaris \cite{mm-thesis}-\cite{mm4} as a property of filters which was detected
by non-low theories, i.e. if $\de$ is not flexible and $Th(M)$ is not low then $M^\lambda/\de$ is not $\lambda^+$-saturated.
Flexibility is presented in \S \ref{s:flexible} below. Stable theories are low. 

The invariant $\mu(\de)$, Definition \ref{mu-defn} below,
describes the minimum size, modulo $\de$, of an unbounded sequence of finite cardinals. 
In Claim \ref{mu-large} we prove that any $\lambda$-flexible ultrafilter 
must have $\mu(\de) = 2^\lambda$. Thus such an ultrafilter will $\lambda^+$-saturate any stable theory, Conclusion \ref{flex-stable}.
By a previous paper \cite{MiSh:996} Theorem 6.4 consistently flexibility does not imply saturation of the 
minimum unstable theory,
the random graph; so this is best possible. 

\step{Preventing future flexibility during an ultrafilter construction.}
In Claim \ref{f2} we prove directly that if, at any point in a construction by independent functions the cardinality of the range of 
the remaining independent family is strictly smaller than the index set, then essentially (i.e. after ``consuming'' one more function) 
no subsequent ultrafilter can be flexible. This gives a point of leverage 
for proving non-saturation.  

\br
The core of the paper contains three constructions. 

\step{Preventing future saturation of any unstable theory.}
In the first construction, Theorem \ref{claim-new}, we show how to build a regular filter $\de$, at the cost of a single independent function $g_*$, 
so that no subsequent ultrafilter saturates the theory of the random graph. 
As the random graph is minimum among the unstable theories in Keisler's order, this shows that no subsequent ultrafilter will
saturate \emph{any} unstable theory. 

This is a theorem in the spirit of Claim \ref{f2} just discussed, that is, a technique which allows one to construct ultrafilters which realize
certain types and omit others by ensuring that the ``omitting types'' half of the construction is already ensured at some bounded
point in the construction.  Now, it has long been known how to construct an \emph{ultrafilter} on $\lambda > 2^{\aleph_0}$ 
which saturates precisely the stable theories,
essentially by organizing the transfinite construction of the ultrafilter so that $\mu(\de)$ is large but $\lcf(\aleph_0)$ is small. 
[In the language of \S \ref{s:constr}, begin with some regular $\delta > \aleph_0$ and a 
$(\lambda, \aleph_0)$-good triple $(I, \de_0, \gee)$ where
$\gee \subseteq {^I \aleph_0}$, $|\gee| = 2^\lambda$. Enumerate $\gee$ 
by an ordinal divisible by $2^\lambda$ and with cofinality $\delta$, 
and apply Fact \ref{fact-lcf}. See \cite{Sh:a} VI.3.12 p. 357 and VI.4.8 p. 379.] 
However, in such constructions the coinitiality
of $\aleph_0$ in the ultrapower mirrors the cofinality of the ultrafilter construction. The construction here, by contrast, 
ensures failure of saturation in any future ultrapower
long before the construction of an ultrafilter is complete.

To prove the theorem, using the language of \S \ref{s:constr}, we begin with $(I, \de_0, \gee)$ a 
$(\lambda, \mu)$-good triple with $\mu^+ < \lambda$, $M \models \trg$ where $R$ denotes the edge relation. 
We unpack the given independent
function $g_* \in \gee$ so it is a sequence $\langle f^*_\epsilon : \epsilon < \mu^+\rangle$, such that 
$(I, \de_0, \gee \setminus \{g_*\} \cup \{ f^*_\epsilon : \epsilon < \mu^+ \})$ is good. 
We then build $\de \supseteq \de_0$ in an inductive construction of length $\mu^+$, consuming the functions $f^*_\epsilon$. 
At each inductive step $\beta$, we ensure that
$f^*_{2\beta}, f^*_{2\beta + 1}$ are $R$-indiscernible to certain distinguished functions $f: I \rightarrow M$, and that 
$f^*_{2\beta}, f^*_{2\beta+1}$ are unequal to each other and to all $f^*_\gamma$, $\gamma < 2\beta$. The structure of the induction
ensures that all functions from $I$ to $M$ are equivalent modulo the eventual filter $\de$ to one of the distinguished functions.
Thus for any subsequent ultrafilter $\de_* \supseteq \de$, $M^I/\de_*$ will omit the type of an element connected to 
$f^*_\gamma$ precisely when $\gamma$ is even, and so fail to be $\mu^{++}$-saturated. 

As a corollary, we have in ZFC that
$\lcf(\aleph_0, \de)$ may be large without saturating the theory of the random graph. This was shown assuming a measurable cardinal in 
\cite{MiSh:996} Theorem 4.2 and otherwise not known. This is another advantage of Theorem \ref{claim-new}, to
disentangle the cofinality of the construction from non-saturation of the random graph.
This question is of interest as the reverse implication was known: $\lcf(\aleph_0, \de)$ is necessary for saturating some unstable theory. 
As the random graph is minimum among the unstable theories in Keisler's order, our result shows it is necessary but not sufficient.

\step{An ultrafilter which is flexible but not good.}
In the second construction, Theorem \ref{flex-not-good-b}, we prove assuming the existence of a measurable cardinal $\kappa$ (to obtain an $\aleph_1$-complete ultrafilter), that on any $\lambda \geq \kappa^+$ there is a regular ultrafilter
which is $\lambda$-flexible but not $\kappa^{++}$-good. 

Specifically, we first use an inductive construction via families of independent functions to 
produce a ``tailor-made'' filter $D$ on $|I| = \lambda$ which, among other things, is $\lambda$-regular, $\lambda^+$-good, 
admits a surjective homomorphism $h: \mcp(I) \rightarrow \mcp(\kappa)$ such that $h^{-1}(1) = \de$.
Letting $E$ be an $\aleph_1$-complete ultrafilter on $\kappa$, we define an ultrafilter
$\de \supseteq D$ by $\de = \{ A \subseteq I : h(A) \in E\}$, and prove it has the properties desired. Two notable features of this
construction are first, the utility of working with boolean algebras, and second, the contrast with Claim \ref{f2} described above.
This is discussed in Remark \ref{r:compare}.

This result addresses a question of Dow 1975 \cite{dow}, and also improves our previous proof on this subject 
in \cite{MiSh:996} Theorem 6.4. There, it is shown by taking a product of ultrafilters that if 
$\kappa > \aleph_0$ is measurable and $2^\kappa \leq \lambda = \lambda^\kappa$ then there is a regular ultrafilter on $I$, $|I| = \lambda$
which is $\lambda$-flexible but not $(2^\kappa)^+$-good. See also Dow \cite{dow} 3.10 and 4.7, and \cite{MiSh:996} Observation 10.9 for a translation.

\step{Realizing some symmetric cuts without saturating any unstable theory.}
In light of the second author's theorem that any theory with the strict order property is maximal in Keisler's order (\cite{Sh:a}.VI 2.6),
 it is natural to 
study saturation of ultrapowers by studying what combinations of cuts may be realized and omitted in ultrapowers of linear order.
The significance of symmetric cuts is underlined by the connection to $SOP_2$ given in the authors' paper \cite{MiSh:998}.

In the third construction, Theorem \ref{wc-lcf}, assuming the existence of a weakly compact cardinal $\kappa$, we prove that
for $\aleph_0 < \theta = \cf(\theta) < \kappa \leq \lambda$ there is a regular ultrafilter $\de$ on $I$, $|I| = \lambda$ such that
$\lcf(\aleph_0, \de) = \theta$ but $(\mathbb{N}, <)^I/\de$ has no $(\kappa, \kappa)$-cuts. 
That is, we build an ultrafilter to these specifications using an independent family $\eff$ of functions with range $\aleph_0$
(so note that the ultrafilter will not be flexible). $\lcf(\aleph_0, \de) \leq \lambda$ implies that $\de$ will fail to
saturate any unstable theory. 

We now briefly describe the structure of the proof.
The construction will have two constraints. On one hand,
we would like the lower cofinality of $\aleph_0$ to be small, equal to $\theta$. 
We can control this in the known way, i.e. by enumerating the steps
in our construction by an ordinal $\delta$ with cofinality $\theta$, and ensuring that we continuously ``consume'' elements of
$\eff$ in such a way that at the end of stage $\eta$ all sets supported by $\eff_\eta \subseteq \eff$ have been decided, and
$\bigcup_{\eta < \delta} \eff_\eta = \eff$. On the other hand, we would like to ensure that the ultrafilter realizes
 all $(\kappa, \kappa)$-pre-cuts (recall the convention on ``cut'' versus ``pre-cut'' in \ref{conventions}). 
Accomplishing this requires two things. 

The first is to ensure that any pre-cut in $(\mathbb{N}, <)^I/\de$ is already a consistent
partial $D_\eta$-type for some $\eta < \delta$, i.e. it is ``already a type'' at some bounded stage in the construction and
thus we will have enough room to try to realize it. Roughly speaking, 
we assign each formula $\vp_\alpha = a^1_\alpha < x < a^2_\alpha$ in the type to the minimum
$\eta < \delta$ such that $X_\alpha = \{ t \in I : M \models \exists x \vp_\alpha(x ; a^1_\alpha[t], a^2_\alpha[t]) \}$ 
is supported by $\eff_\eta$; without loss of generality 
the range of this function is $\theta = \cf(\delta)$, and by weak compactness it is constant on a cofinal subset of the cut.

The second, more substantial task is to show, at a given inductive step in the construction, that a given pre-cut can be realized. 
By the previous paragraph, we may assume that for each formula $\vp_\alpha$ in the type, the set $X_\alpha$ belongs to the current filter.
Thus we have a distribution for the type in hand  
and we would like to extend the filter to include a multiplicative refinement.
In particular, we are obliged to choose a suitable refinement of each $X_\alpha$. 
To do this, we build what are essentially indiscernible sequences (of countable sequences of elements) 
in the underlying boolean algebra, one for each $\alpha$. 
We do this so that row $\beta$ in indiscernible sequence $\alpha$ is a partition of the boolean algebra
on which the set of solutions to $\{ t \in I : M \models \exists x ( a^1_\alpha[t] < a^1_\beta[t] < x < a^2_\beta[t] < a^1_\alpha[t] ) \}$
is based. (Here $\alpha < \beta$ and $\alpha, \beta$ range over some cofinal sequence in $\kappa$.)
The templates for such sequences are extracted using the strong uniformity we have available on $\kappa$. 
We then show how to obtain a multiplicative refinement by generically extending each such sequence one additional step. 

To finish, we indicate how to avoid large cardinal hypotheses for some related results.

\section{Background: Flexibility, independent families of functions, boolean algebras} \label{s:background}

\subsection{Basic definitions} \label{s:basic} We define regular filters, good filters and Keisler's order.

\begin{defn} \emph{(Regular filters)} \label{regular}
A filter $\de$ on an index set $I$ of cardinality $\lambda$ is said to be \emph{$\lambda$-regular}, or simply 
\emph{regular}, if there exists a $\lambda$-regularizing family $\langle X_i : i<\lambda \rangle$, which means that:
\begin{itemize}
 \item for each $i<\lambda$, $X_i \in \de$, and
 \item for any infinite $\sigma \subset \lambda$, we have $\bigcap_{i \in\sigma} X_i = \emptyset$
\end{itemize}
Equivalently, for any element $t \in I$, $t$ belongs to only finitely many of the sets $X_i$. 
\end{defn}

\begin{defn} \emph{(Good ultrafilters, Keisler \cite{keisler-1})}
\label{good-filters}
The filter $\de$ on $I$ is said to be \emph{$\mu^+$-good} if every $f: \fss(\mu) \rightarrow \de$ has
a multiplicative refinement, where this means that for some $f^\prime : \fss(\mu) \rightarrow \de$,
$u \in \fss(\mu) \implies f^\prime(u) \subseteq f(u)$, and $u,v \in \fss(\mu) \implies
f^\prime(u) \cap f^\prime(v) = f^\prime(u \cup v)$.

Note that we may assume the functions $f$ are monotonic.

$\de$ is said to be \emph{good} if it is $|I|^+$-good.
\end{defn}

Keisler proved the existence of $\lambda^+$-good countably incomplete ultrafilters on $\lambda$ assuming $2^\lambda = \lambda^+$. 
Kunen \cite{kunen} gave a proof in ZFC, which introduced the technique of independent families of functions.

The crucial model-theoretic property of regularity is the following: for a regular ultrafilter on $\lambda$ and
a complete countable theory $T$, $\lambda^+$-saturation of the ultrapower $M^\lambda/\de$ 
does not depend on the choice of base model $M$. 

\begin{thm-lit} \label{backandforth} \emph{(Keisler \cite{keisler} Corollary 2.1 p. 30; see also Shelah \cite{Sh:c}.VI.1)}
Suppose that $M_0 \equiv M_1$, the ambient language is countable $($for simplicity$)$, and
$\de$ is a regular ultrafilter on $\lambda$.
Then ${M_0}^\lambda/\de$ is $\lambda^+$-saturated iff ${M_1}^\lambda/\de$
is $\lambda^+$-saturated.
\end{thm-lit} 

Thus Keisler's order is genuinely a statement about the relative complexity of 
[complete, countable] theories, independent of the choice of base models $M_1, M_2$:

\begin{defn} \label{keisler-order} \emph{(Keisler 1967 \cite{keisler})} 
Let $T_1, T_2$ be complete countable first-order theories. 
\begin{enumerate}
\item \emph{$T_1 \tlf_\lambda T_2$} when for all regular ultrafilters $\de$ on $\lambda$,
all $M_1 \models T_1$, all $M_2 \models T_2$, if $M^{\lambda}_2/\de$
is $\lambda^+$-saturated then $M^{\lambda}_1/\de$ is $\lambda^+$-saturated. 
\item \emph{(Keisler's order)} 
$T_1 \tlf T_2$ if for all infinite $\lambda$, $T_1 \kleq_\lambda T_2$.
\end{enumerate}
\end{defn}

\br

An account of current work on Keisler's order is given in the companion paper \cite{MiSh:996}. 
As that history is not needed for reading the current paper,
we quote the most relevant theorem in \S \ref{s:review} below and refer interested readers to \cite{MiSh:996} sections 1-4. 

\br
\subsection{A translation between model theory and ultrafilters} \label{s:review}
The following theorem of known correspondences between properties of regular ultrafilters and properties of 
first-order theories is quoted from the companion paper \cite{MiSh:996}. 
Conditions (1), (2) are defined in \S \ref{s:flex-mu}, (3) in Convention \ref{conventions}[\ref{cv:good}] below, (4) in \S \ref{s:flexible}, 
(5) in \cite{mm4} (not used here) and (6) in \S \ref{s:basic}.   

\begin{thm-lit} \label{formula-corr} \emph{(Malliaris and Shelah \cite{MiSh:996} \S 4 Theorem F)} 
In the following table, for rows \emph{(1),(3),(5),(6)} the regular ultrafilter $\de$ on $\lambda$ fails to have 
the property in the left column if and only if it omits a type in every formula with the property in the right column.
For rows \emph{(2)} and \emph{(4)}, if $\de$ fails to have the property on the left then it omits a type 
in every formula with the property on the right.

\br
\begin{tabular}{lcl}
\br \textbf{Set theory: properties of filters} & & \textbf{Model theory: properties of formulas} \\
 \emph{(1)} $\mu(\de) \geq \lambda^+$ & & A. finite cover property \\
 \emph{(2)} $\lcf(\aleph_0, \de) \geq \lambda^+$ & ** & B. order property \\
 \emph{(3)} good for $T_{rg}$ & & C. independence property \\
 \emph{(4)} flexible, i.e. $\lambda$-flexible	& ** & D. non-low \\
 \emph{(5)} good for equality & & E. $TP_2$ \\
 \emph{(6)} good, i.e. $\lambda^+$-good&  & F. strict order property \\ 
\end{tabular}
\end{thm-lit}

\begin{proof} 
(1) $\leftrightarrow$ (A) Shelah \cite{Sh:c}.VI.5. Note that the f.c.p. was defined in Keisler \cite{keisler}.

(2) $\leftarrow$ (B) Shelah \cite{Sh:c}.VI.4.8.

(3) $\leftrightarrow$ (C) Straightforward by q.e., see \cite{mm4}, and \cite{mm5} for the more general phenomenon.  

(4) $\leftarrow$ (D) Malliaris \cite{mm-thesis}, see \S \ref{s:flexible} below.

(5) $\leftrightarrow$ (E) Malliaris \cite{mm4} \S 6, which proves the existence of a Keisler-minimum $TP_2$-theory, the theory $T^*_{feq}$ of
a parametrized family of independent (crosscutting) equivalence relations. 

(6) $\leftrightarrow$ (F) Keisler characterized the maximum class by means of good ultrafilters. 
Shelah proved in \cite{Sh:c}.VI.2.6 that any theory with the strict order property is maximum
in Keisler's order. (In fact, $SOP_3$ suffices \cite{Sh500}). 
A model-theoretic characterization of the maximum class is not known.
\end{proof}

The known arrows between properties (1)-(6) are given in \cite{MiSh:996} Theorem 4.2. In particular, 
the arrow (4) $\rightarrow$ (1) is from \S \ref{s:flex-mu} below.

\subsection{Flexible filters} \label{s:flexible}
We now give background on flexible filters, a focus of this paper. 
Flexible filters were introduced in Malliaris \cite{mm-thesis} and \cite{mm4}.
In the context of investigations into saturation of regular ultrapowers,
a natural question is whether and how first-order theories are sensitive to the sizes
of regularizing families:

\begin{defn} \emph{(Flexible ultrafilters, Malliaris \cite{mm-thesis}, \cite{mm4})}
\label{flexible}
We say that the filter $\de$ is $\lambda$-flexible if for any $f \in {^I \mathbb{N}}$ with
$n \in \mathbb{N} \implies n <_{\de} f$, we can find $X_\alpha \in \de$ for $\alpha < \lambda$ such that 
for all $t \in I$
\[ f(t) \geq | \{ \alpha : t \in X_\alpha \}|  \]
Informally, given any nonstandard integer, we can find a $\lambda$-regularizing family below it.
\end{defn}

Alternatively, one could say that in $(\mch(\aleph_0), \epsilon)^\lambda/\de$, any $\lambda$ elements belong to a
pseudofinite set of arbitrarily small size [in the sense of the proof of Claim \ref{f2}]. 
It is useful to know that flexible is equivalent to the set-theoretic ``o.k.'' 
(see\cite{MiSh:996} Appendix and history there).
 
The importance of flexibility for our construction comes from the following lemma, which gives one of the arrows in 
Theorem \ref{formula-corr} above.

\begin{lemma} \emph{(Malliaris \cite{mm4} Lemma 8.7)} \label{mm-thesis-lemma}  
\label{flexible-low}
Let $T$ be non-low, $M \models T$ and let $\de$ be a $\lambda$-regular ultrafilter on $I$, $|I| =\lambda$ which is not $\lambda$-flexible. Then
$N:=M^\lambda/\de$ is not $\lambda^+$-saturated.  
\end{lemma}

\begin{cor} \label{good-flexible} Flexibility is a non-trivial hypothesis, i.e.
\begin{enumerate}
\item Not all regular ultrafilters are flexible.
\item Some regular ultrafilters are flexible. In particular, 
if $\de$ is a regular ultrafilter on $\lambda$ and $\de$ is $\lambda^+$-good then $\de$ is $\lambda$-flexible. 
\end{enumerate}
\end{cor}

\begin{proof}
(1) By the fact that there is a minimum class in Keisler's order which does not include the non-low theories. 

(2) One can prove this directly, or note that since a $\lambda^+$-good 
ultrafilter on $\lambda$ saturates \emph{any} countable theory, in particular 
any non-low theory, it must be flexible by Lemma \ref{flexible-low}. 
\end{proof}

We conclude by describing the known model-theoretic strength of flexibility.

\begin{fact}
If $\de$ is not flexible then $M^\lambda/\de$ is not $\lambda^+$-saturated whenever $Th(M)$ is not simple or simple but not low. 
\end{fact}

\begin{proof}
By Lemma \ref{mm-thesis-lemma} in the case where $T$ is not low; by Malliaris and Shelah
\cite{MiSh:998} in the case where $T$ has $TP_1$; by Malliaris \cite{mm4} in the case where $T$ has $TP_2$.
\end{proof}

\subsection{Independent families of functions and Boolean algebras} \label{s:constr}
We now give some preliminaries, notation and definitions for the construction of ultrafilters. 
We will follow the notation of \cite{Sh:c} Chapter VI, Section 3, and further details may be found there. 

We will make extensive use of independent families, which provide a useful gauge of the freedom left when building filters.

\begin{defn}
Given a filter $\de$ on $\lambda$, we say that a family $\eff$ of functions from $\lambda$ into $\lambda$ is
\emph{independent $\mod \de$} if for every $n<\omega$, distinct $f_0, \dots f_{n-1}$ from $\eff$ and choice of $j_\ell \in \rn(f_\ell)$,
\[ \{ \eta < \lambda ~:~ \mbox{for every $i < n, f_i(\eta) = j_i$} \} \neq \emptyset ~~~ \operatorname{mod} \de \]
\end{defn}

\begin{thm-lit} \label{thm-iff} \emph{(Engelking-Karlowicz \cite{ek} Theorem 3, Shelah \cite{Sh:c} Theorem A1.5 p. 656)}
For every $\lambda \geq \aleph_0$ there exists a family $\eff$ of size $2^\lambda$ with each $f \in \eff$ from 
$\lambda$ onto $\lambda$ such that $\eff$ is independent
modulo the empty filter \emph{(}alternately, by the filter generated by $\{ \lambda \}\emph{)}$. 
\end{thm-lit}

\begin{cor}
For every $\lambda \geq \aleph_0$ there exists a regular filter $\de$ on $\lambda$ and a family $\eff$ of size $2^\lambda$ which is independent
modulo $\de$.
\end{cor}

The following definition describes the basic objects of our ultrafilter construction.

\begin{defn} \emph{(Good triples, \cite{Sh:c} Chapter VI)} \label{good-triples}
Let $\lambda \geq \kappa \geq \aleph_0$, $|I| = \lambda$, $\de$ a regular filter on $I$, and 
$\gee$ a family of functions from $I$ to $\kappa$.
\begin{enumerate}
\item Let $\fin(\gee) = \{ h ~: ~h: [\gee]^{<\aleph_0} \rightarrow \kappa ~\mbox{and}~ g \in \dom(g) \implies h(g) \in \rn(g) \} $
\item For each $h \in \fin(\gee)$ let
\[ A_h = \{ t \in I ~:~ g \in \dom(g) \implies g(t) = h(g) \} \]
\item Let $\fin_s(\gee) = \{ A_h : h \in \fin(\gee) \}$ 
\item We say that triple $(I, \de, \gee)$ is $(\lambda, \kappa)$-pre-good when $I$, $\de$, $\gee$ are as given, and
for every $h \in \fin(\gee)$ we have that $A_h \neq \emptyset \mod \de$.
\item We say that $(I, \de, \gee)$ is $(\lambda, \kappa)$-good when $\de$ is maximal subject to being pre-good.
\end{enumerate}
\end{defn}

\begin{fact} \label{good-dense}
Suppose $(I, \de, \gee)$ is a good triple. Then $\fin_s(\gee)$ is dense in $\mcp(I) \mod \de$. 
\end{fact}

\begin{defn} \emph{(for more on Boolean algebras, see \cite{Sh:c} Definition 3.7 p. 358)} \label{d:ba-various}
\begin{enumerate}
 \item  A \emph{partition} in a Boolean algebra is a maximal set of pairwise disjoint non-zero elements. 
 \item For a Boolean algebra $B$, $CC(B)$ is the first regular cardinal $\lambda \geq \aleph_0$ such that
every partition of $B$ has cardinality $< \lambda$. 
\item An element $\mb$ of a Boolean algebra is \emph{based} on a partition $\mpp$ if $\maa \in \mpp$ implies
$\maa \subseteq \mb$ or $\mb \cap \maa = 0$. 
\item An element of a Boolean algebra $B$ is \emph{supported} 
by a set $\mpp$ of elements of $B$ if it is based on some partition $\mpp$ of $B$ with $\mpp \subseteq P$.
\end{enumerate}
\end{defn}

For completeness, we quote the following fact which will be used in the proof of Claim \ref{f2}. It
explains how the range of the independent families available directly reflects the amount of freedom 
(specifically, the size of a maximal disjoint family of non-small sets) remaining in the construction of the filter. 

\begin{fact} \label{mcc} \emph{(\cite{Sh:c} Claim 3.17(5) p. 359)}
Suppose $(I, \de, \gee)$ is a good triple, where for at least one $g \in \gee$,
$|\rn(g)| \geq \aleph_0$, or alternatively $\{ g \in \gee : | \rn(g) | > 1 \}$ is infinite. 
Then $CC(B(D))$ is the first regular $\nu > \aleph_0$ such that $g \in \gee \implies
|\rn(g)| < \nu$. Moreover, if $\lambda \geq \nu$ is regular, $A_i \neq \emptyset \mod \de$ for $i<\lambda$, then 
there is $S \subseteq \lambda$, $|S| = \lambda$ such that for $n<\omega$ and distinct $i(\ell) \in S$, we have that
$\bigcap_{\ell<n} A_{i(\ell)} \neq \emptyset \mod \de$. 
\end{fact}

Finally, the following lemma gives the necessary scaffolding and guarantees 
that the end product of our construction will be an ultrafilter.

\begin{fact} \label{uf} \emph{(\cite{Sh:c} Lemma 3.18 p. 360)}
Suppose that $\de$ is a maximal filter modulo which $\eff \cup \gee$ is independent, $\eff$ and $\gee$ are disjoint,
the range of each $f \in \eff \cup \gee$ is of cardinality less than $\operatorname{cof}(\alpha)$, $\operatorname{cof}(\alpha) > \aleph_0$,
$\eff = \bigcup_{\eta < \alpha} \eff_\eta$, the sequence $\langle \eff_\eta : \eta < \alpha \rangle$ 
is increasing, and let $\eff^\eta = \eff \setminus \eff_\eta$.
Suppose, moreover, that $D_\eta$ ($\eta<\alpha$) is an increasing sequence of filters which satisfy:

\begin{enumerate}
\item[(i)] Each $D_\eta$ is generated by $\de$ and sets supported $\mod \de$ by $\fin_s(\eff_\eta \cup \gee)$.
\item[(ii)] $\eff^\eta \cup \gee$ is independent modulo $D_\eta$. 
\item[(iii)] $D_\eta$ is maximal with respect to \emph{(i), (ii)}. 
\end{enumerate}
Then
\begin{enumerate}
\item $D^* := \bigcup_{\eta < \alpha} D_\eta$ is a maximal filter modulo which $\gee$ is independent.
\item If $\gee$ is empty, then $D^*$ is an ultrafilter, and for each $\eta < \alpha$, \emph{(ii)} is satisfied whenever $D_\eta$ is
non-trivial and satisfies \emph{(i)}.
\item If $\eta < \alpha$ and we are given $D^\prime_\eta$ satisfying \emph{(i), (ii)} we can extend it to a filter satisfying
\emph{(i), (ii), (iii)}. 
\item If $f \in \eff^\eta$ then $\langle f^{-1}(t) /D_\eta : t \in Range(f) \rangle$ is a partition in $B(D_\eta)$.  
\end{enumerate}
\end{fact}

A useful consequence of this machinery is the following general principle.
It is rigorously developed in the proof of \cite{Sh:c} Theorem VI.3.12 p. 357-366. Specifically, 
the formal statement is \cite{Sh:c} Claim 3.21 p. 363. 

\begin{fact} \label{fact-lcf} \emph{(Cofinality of the construction and lower-cofinality of $\aleph_0$, \cite{Sh:c} p. 363)}
In the notation of Fact \ref{uf} [note restrictions on cofinality of $\alpha$ there], 
suppose that $\gee = \emptyset$ and that we are given a sequence of functions
$f_\beta \in \eff^\eta \setminus \eff^{\eta+1}$ such that for every $n < \aleph_0$, 
$\{ t : n < f_\beta(t) < \aleph_0 \} \in \de_{\eta+1}$. Then for every ultrafilter
$\de_* \supseteq \de$, we have that $\lcf(\aleph_0, \de_*) = \cf(\alpha)$ and this is exemplified by 
the sequence $f_\beta/\de_*$ for $\beta < \alpha$.
\end{fact}

\br
\subsection{Conventions}

\begin{conv} \emph{(Conventions)} \label{conventions}
\begin{enumerate}
\item \label{c:filter} 
When $\de$ is a filter on $I$ and $\mathcal{X} \subseteq \mcp(I)$, by $\de \cup \mathcal{X}$ we will mean the filter generated by 
$\de \cup \mathcal{X}$. By $(\de \cup \mathcal{X})^+$ we mean the sets which are nonzero modulo the filter generated by 
$\de \cup \mathcal{X}$.
\item Throughout, tuples of variables may be written without overlines, that is: 
when we write $\vp = \vp(x;y)$, neither $x$ nor $y$ are necessarily assumed to have length 1, but are finite.
\item For transparency, all languages (=vocabularies) are assumed to be countable. 
\item When $M^I/\de$ is an ultrapower we refer to $M$ as the ``base model.''
\item By ``$\de$ saturates $T$'' we will always mean:
$\de$ is a regular ultrafilter on the infinite index set $I$, $T$ is a countable complete first-order theory and for any $M \models T$,
we have that $M^I/\de$ is $\lambda^+$-saturated, where $\lambda = |I|$.
\item \label{cv:good} We will also say that the ultrafilter $D$ is ``good'' \lp or: ``not good''\rp~ for the theory $T$ to mean that $D$ saturates 
\lp or: does not saturate\rp~ the theory $T$.
\item A partial type (\emph{or} quantifier-free complete $1$-type) in a model $M$ of the theory of partial order given by some
pair of sequences $( \langle a_\alpha : \alpha < \kappa_1 \rangle, \langle b_\beta : \beta < \kappa_2 \rangle)$ with
$\alpha < \alpha^\prime < \kappa_1, \beta < \beta^\prime < \kappa_2 \implies M \models a_\alpha < a_{\alpha^\prime} < b_{\beta^\prime} < b_\beta$,
which may or may not have a realization in $M$, is called a \emph{pre-cut}. 
Our main case is $M^I/\de$ for $M$ a linear order and $\de$ a filter on $I$.
\item We reserve the word \emph{cut} in models of linear order for pre-cuts which are omitted types.
\end{enumerate}
\end{conv}

This concludes the front matter. 

\br

\section{$\de$ $\lambda$-flexible on $I$, $f \in {^I\mathbb{N}}$ $\de$-nonstandard implies $|\prod_{s \in I} 2^{f(s)}/\de| = 2^\lambda$}
\label{s:flex-mu}

In this section we prove Claim \ref{mu-large}. This result established the arrow \cite{MiSh:996} Theorem 4.2 (4) $\rightarrow$ (1)
[note: (1), (4) are in the notation of \S \ref{s:review} Theorem \ref{formula-corr} above] as well as Conclusion \ref{flex-stable}. 
We first state a definition and a theorem. 

\begin{defn} \label{mu-defn} \emph{(\cite{Sh:c} Definition III.3.5)}
Let $\de$ be a regular ultrapower on $\lambda$.
\[  \mu(\de) :=  \operatorname{min} \left\{ \rule{0pt}{15pt}
\prod_{t<\lambda}~ n[t] /\de  ~: ~ n[t] < \omega, ~\prod_{t<\lambda}
~n[t]/\de \geq \aleph_0 \right\} \]
be the minimum value of the product of an unbounded sequence of cardinals
modulo $\de$.
\end{defn}

\begin{thm-lit} \emph{(Shelah, \cite{Sh:c}.VI.3.12)} \label{mu-theorem}
Let $\mu(\de)$ be as in Definition \ref{mu-defn}. Then for any infinite
$\lambda$ and
$\nu = \nu^{\aleph_0} \leq 2^{\lambda}$ there exists a regular ultrafilter
$\de$ on $\lambda$ with $\mu(\de) = \nu$.
\end{thm-lit}

We show here that flexibility (i.e. regularity below any nonstandard integer) makes $\mu$ large.

\begin{claim} \label{mu-large}
Let $\de$ be an ultrafilter on $I$, and $f \in {^I\mathbb{N}}$ such that for all $n \in \mathbb{N}$,
$n <_\de f$. If $\de$ is $\lambda$-regular below $f$ then $|\prod_{s \in I} 2^{f(s)}/\de| = 2^\lambda$.
\end{claim}

\begin{proof}
In one direction, $|\prod_{s \in I} 2^{f(s)}/\de| \leq 2^\lambda$
as $f \leq {\aleph_0}^I \mod \de = 2^{|I|}$ by regularity of $\de$, see \cite{MiSh:996}, Fact 5.1. 

In the other direction, let $\langle X_i : i < \lambda \rangle$ be a regularizing family below $f$, so each $X_i \subseteq I$. 
Now for any $A \subseteq \lambda$, define $g_A \in {^I\mathbb{N}}$ by:
\[ \mbox{if}~ s \in I ~ \mbox{let}~ A_s = \{ i < \lambda : s \in X_i \} \]
so for each $s \in I$, $|A_s| \leq f(s)$, and furthermore let
\[  g_A(s) = \sum \{ 2^{|A_s \cap i|} : i \in A \cap A_s \} \]
noting that the range can be thought of as a number in binary representation, thus if $A \cap A_s$ differs from $A^\prime \cap A_s$
we have $g_A(s) \neq g_{A^\prime}(s)$: at least one place the ``binary representations'' are different.

First notice that:
\[ g_A(s) \leq \sum \{ 2^\ell : \ell < |A_s| \} = \sum \{ 2^\ell : \ell < f(s) \} < 2^{f(s)} \]

Second, suppose $B_1\neq B_2$ are subsets of $\lambda$. Without loss of generality, $B_1 \not\subseteq B_2$, and let
$j \in B_1 \setminus B_2$. If $s \in X_j$, $g_{B_1}(s), g_{B_2}(s)$ differ as noted above, so as $X_j \in \de$ we conclude:
\[ \mbox{if} ~ B_1 \neq B_2 ~ \mbox{are subsets of $\lambda$ then}~ g_{B_1} \neq g_{B_2} \mod \de \]

These two observations complete the proof.   
\end{proof}

\begin{concl} \label{flex-stable}
If $\de$ is a flexible ultrafilter on $I$ then $\mu(\de) \geq \lambda^+$, and thus $\de$ saturates any
stable theory.
\end{concl}

\begin{proof}
It is known that $\de$ saturates all countable stable theories if and only if $\mu(\de) \geq \lambda^+$ 
(see \cite{MiSh:996}, \S 4, Theorem F). 
Now $f \in {^I\mathbb{N}}$ is $\de$-nonstandard if and only if $\prod_{s \in I} 2^{f(s)}/\de$ is nonstandard,
since we can exponentiate and take logarithms in an expanded language.
Thus, given any nonstandard $f$, apply Claim \ref{mu-large} to $\operatorname{log} f$ to conclude it is large. 
\end{proof}

\section{If $(I, \de, \gee)$ is $(\lambda, \mu)^+$-good for $\mu < \lambda$ then no subsequent ultrafilter is flexible}

As a warm-up to preventing saturation of the random graph in Section \ref{s:trg}, here
we show how to ensure directly, at some bounded stage in an ultrafilter construction, that no subsequent ultrafilter will be flexible.
Recall Fact \ref{mcc}.  The idea of building in a failure of saturation via independent functions will be substantially extended
in \S \ref{s:trg}.
Note that in \cite{MiSh:999} we show that this statement can be derived, by a different method, from a result in \cite{Sh:a} on good filters.

\begin{claim} \label{f2}  
\emph{(Preventing flexibility)}
\begin{enumerate}
\item If $(I, \de, \gee)$ is $(\lambda, \mu)$-good, $g \in \gee$, $|\rn(g)| \geq \aleph_0$ \el without loss of generality $\rn(g) \supseteq \mathbb{N}$\er
$~$  then every ultrafilter $\de_*$ on $I$ extending $\de \cup \{ \{ s \in I : n < g(s) < \omega \} : n < \omega \}$
is not $\mu^+$-flexible, witnessed by $g/\de_*$.
\item More generally, if:
\begin{itemize}
\item $\mu \leq \lambda$ is regular 
\item $(I, \de, \gee)$ is $(\lambda, {<\mu})$-good, meaning that $g \in \gee \implies \rn(g)$ is an ordinal ${< \mu}$
\item $g \in \gee$, $|\rn(g)| \geq \aleph_0$ \el without loss of generality $\rn(g) \supseteq \mathbb{N}$\er
\end{itemize}
then every ultrafilter $\de_*$ on $I$ extending $\de \cup \{ \{ s \in I : n < g(s) < \omega \} : n < \omega \}$
is not $\mu$-flexible, witnessed by $g/\de_*$.
\end{enumerate}
\end{claim}

\begin{proof}
(1) Let $g$ and $\de$ be given.    
Let $g_* \in {^I\mathbb{N}}$ be given by: $g_*(s) = g(s)$ 
if $g(s) \in \mathbb{N}$, and $g_*(s) = 0$ otherwise. Let $\ee$ be the filter generated by
\[ \de \cup \{ \{ s \in I : g_*(s) > n\} : n \in \omega \} \] 
Note the definition of
``$(I, \de, \gee)$ is good'' ensures that $\ee$ is a non-trivial filter, recalling \ref{conventions}(2). 

For each $i<\mu^+$, let $\ci = {^I\{i\}}$ be constantly $i$. 
Let $\de_* \supseteq \ee$ be any ultrafilter on $I$ extending $\ee$. We now ask:
\[ \mbox{Does $(\mathcal{H}(\lambda^+), \epsilon)^I/\de_* \models$ ``there is a set $b/\de_*$ with $g/\de_*$ 
members s. t. for every $i < \mu^+$, $\ci/\de_* \in b$''?} \]

Assume towards a contradiction that it does, witnessed by $b/\de_*$ for some given $b \in {^I(\mathcal{H}(\lambda^+))}$.
Since $\de_*$ is regular, we may assume without loss of generality that the projections to the base model are
finite, i.e. that for each $s \in I$, $b[s] \in [\lambda]^{\leq g_*(s)}$. 

For each $i<\mu^+$, let $B_{b,i} = \{ s \in I : i \in b[s] \}$ be the set on which the constant function $\ci$ is in the
projection of $b$. For each $i$, $B_{b,i} \neq \emptyset \mod \de$
as it belongs to an ultrafilter extending $\de$. Thus, since $(I, \de, \gee)$ is a good triple, 
there is $A_{h_i}$ for $h_i \in \fin(\gee)$ such that $A_{h_i} \subseteq B_{b,i}$. 

Clearly, any two constant functions $\ci, \textbf{f}_j$ are everywhere distinct. For each $n<\omega$, let 
\[ u_n = \{ i<\mu^+ : B_{b,i} \in (\de \cup g^{-1}_*(n))^+ \} \]  
Note that any such $B_{b,i}$ will contain $A_{h_{n,i}} \mod \de$ for some $h_{n,i} \in \fin(\gee)$ which includes the
condition that $g=n$, i.e. the condition that $b$ have exactly $n$ elements.

Suppose first that for some $n$, $|u_n| \geq \mu^+$.
By Fact \ref{mcc}, we know that since $CC(B(\de)) = \mu^+$ by the assumption on $\gee$, 
there are $i_1 < \dots < i_{n+1}$ in $u_n$ 
such that $h := \bigcup \{ h_{n,i_\ell} : 1 \leq \ell \leq n+1 \} \in \fin(\gee)$. By choice of the $h_{n,i}$
we have that $A_h \cap g^{-1}(n) \neq \emptyset \mod \de$, i.e. there is a nonempty set on which $n+1$ everywhere distinct
elements each belong to a set of cardinality $n$, a contradiction. 

Thus for each $n<\omega$, we must have $|u_n| \leq \mu$. Hence 
$u := \bigcup \{ u_n : n<\omega \}$ has cardinality $\leq \mu$. 
On the other hand, if $i \in \mu^+ \setminus u$ then by definition for each $n<\omega$, $B_{b,i} \cap g^{-1}_*(\{n\}) = \emptyset$
$\mod \de$. Hence for $g$, $B_{b,i} \cap \{ t : g(t) \geq \omega \} = \emptyset \mod \de$. 
Since $g \in \gee$ and we had assumed $\de$ maximal modulo which $\gee$ was independent, this implies that already
$B_{b,i} = \emptyset \mod \de$.  Hence for any $i \in \mu^+ \setminus u$, and for any ultrafilter $\de_* \supseteq \ee$, 
\[ (\mathcal{H}(\lambda^+), \epsilon)^I/\de_* \models \mbox{``$\ci/{\de_*} \notin b/{\de_*}$''} \] 
We have shown that in any such ultrafilter, $g_*/\de_* > \aleph_0$ but there is no $\mu^+$-regularizing set of size $\leq g_*/\de_*$.
This completes the proof.

(2) Same proof, since in this case we can still apply Fact \ref{mcc}.
\end{proof}

\begin{rmk} \label{r:compare}
Compare Claim \ref{f2} to the main theorem of Section \ref{s:meas}.
Claim \ref{f2} shows that if in some point in the construction of an ultrafilter via families of independent functions, we reach a point where the $CC$ of the remaining Boolean algebra is small, then \emph{after adding one more function}, no subsequent ultrafilter can be flexible. 
However, this is not a fact about $CC(B(D))$ alone. 
Theorem \ref{flex-not-good-b} constructs a flexible, not good ultrafilter by means of a quotient: 
the key step there is to begin with a filter $\de$ on $\lambda$ such that there is a Boolean algebra homomorphism 
$h: \mcp(I) \rightarrow \mcp(\kappa)$ with $h^{-1}(1) = \de$, 
where $\kappa < \lambda$. We then take the preimage of a complete ultrafilter on $\kappa$ to complete
the construction. This second ultrafilter is flexible, thanks to the completeness.
\end{rmk}

\br

\section{Omitting types in ultrapowers of the random graph} \label{s:trg}

In this section we show how to prevent saturation directly in ultrapowers of the random graph. 
We write $\trg$ for the theory of the random graph and consider models of $\trg$, unless otherwise stated.

\br
\step{Step 0. Preliminary Discussion to Theorem \ref{claim-new}.}
In this step we assume $\cf(\kappa) > \mu$.

Suppose that $\langle (I, \de_{\alpha}, \gee_{\alpha}) : \alpha < \kappa \rangle$ is some continuous sequence
of $(\lambda, \mu)$-good triples, where the filters are increasing with $\alpha$ and the families of functions are decreasing with $\alpha$.
Suppose $\de_\kappa = \bigcup_{\alpha < \kappa} \de_\alpha$ is a \emph{filter} (not an ultrafilter) built by such an induction and 
$\gee_\kappa = \bigcap_{\alpha < \kappa} \gee_\alpha \neq \emptyset$. Then by Fact \ref{uf}, as $\cf(\kappa) > \mu$, 
the limit triple $(I, \de_\kappa, \gee_{\kappa})$ is also $(\lambda, \leq\mu)$-good. Write $\gee = \gee_\kappa$ for
this set of functions which remains free. 

Our strategy will be to build a barrier to saturation into any subsequent construction of an ultrafilter $\de \supseteq \de_\kappa$ 
by first constructing the filter $\de_\kappa$, in $\kappa$ steps, 
to have a ``blind spot.'' We now explain what this means. Let $M$ be any model of the random graph.
For any function $g$ from $I$ to $M$ and any element $a \in M$ we may define
\[ A_{g,a} = \{ t : M \models  g(t) R a  \}  \]
Then by definition of ``good triple,'' for each such $g$ and $a$, $A_{g,a}$ belongs, in $\mcp(I)/\de_{\kappa}$.
to the minimal completion of the subalgebra generated by 
\[ \langle f^{-1}(\{\alpha\}) : f \in \gee,  \alpha \in \rn(f) \rangle \]
Since $\cf(\kappa) > \mu$, for each $A_{g,a}$ this will already be true in $\mcp(I)/\de_\alpha$ for some $\alpha < \kappa$. 
Thus at each stage $\beta$ of our induction we will define
\begin{align*}
X_\beta = & \{ g : ~\mbox{$g$ is a function from $I$ to $M$ and for every $a \in M$}\\
                 & \mbox{the set $A_{g,a} = \{ t : M \models  g(t) R a  \}$ belongs, in $\mcp(I)/\de_{\beta}$, }\\
	     & \mbox{ to the minimal completion of the subalgebra generated by } \\
                 &  \langle f^{-1}(\{\alpha\}) : f \in \gee, \alpha \in \rn(f) \rangle \} \\
\end{align*}
The key point of the construction is then to ensure (\ref{claim-new} Step 3 item 7) 
that a distinguished sequence of elements of the ultrapower look alike to all $g \in X_\beta$,
and moreover (in \ref{claim-new} Step 3 item 1) that the resulting triple at $\beta$ is good. Since by the end of the induction
\emph{all} $g$ belong to some $X_\beta$, we will have constructed a sequence of elements of the reduced product
which is effectively indiscernible under any completion to an ultrapower. 
We can then directly find an omitted type in \ref{claim-new} Steps 4-5.

\begin{theorem} \label{claim-new}
If $(A)$ then $(B)$ where:
\begin{enumerate}
\item[$(A)$] We are given $I, \de_0, \gee, g_*, \lambda, \mu$ such that:
\begin{enumerate}
\item $\de_0$ is a regular filter on $I$
\item $(I, \de_0, \gee \cup \{ g_* \})$ is $(\lambda, \mu)$-good
\item $\mu^+ < \lambda$
\item $\rn(g_*) = \mu$
\end{enumerate}
\item[$(B)$] Then there is a filter $\de \supseteq \de_0$ such that:
\begin{enumerate}
\item $(I, \de, \gee)$ is $(\lambda, \mu)$-good
\item if $\de_*$ is any ultrafilter on $I$ extending $\de$ and $M \models \trg$, then
$M^I/\de_*$ is not $\mu^{++}$-saturated. In particular, it is not $\lambda^+$-saturated.
\end{enumerate}
\end{enumerate}
Informally speaking, at the cost of a single function we can prevent future saturation of the theory of the random graph, 
thus of any unstable theory.
\end{theorem}

\begin{proof} The proof will have several steps, within the general framework described at the beginning of this section. 
``Thus of any unstable theory'' is immediate from the fact that $\trg$ is minimum among the unstable theories in Keisler's order. 

\br
\step{1. Background objects.} We fix $M \models \trg$ of size $\mu$ given with some enumeration, 
and $\langle h_\epsilon : \epsilon < \mu^+ \rangle$ an independent family of functions from $\mu$ onto $\mu$. 
In a slight abuse of notation, identify $\mu = \rn(g_*)$ with the domain of $M$ under the given enumeration.

\br
\step{2. The blow-up of $g_*$.} For $\epsilon < \mu^+$, define
$f^*_\epsilon = h_\epsilon \circ g_*$. Then each $f^*_\epsilon$ is a function from $I$ to $\mu$.
Recall from $(B)$(b) that $\gee = \gee^\prime \setminus \{ g_* \}$.
Then letting $\gee_0 = \gee \cup \{ f^*_\epsilon : \epsilon < \mu^+ \}$, we have that
$\gee_0 \subseteq {^I \mu}$ is an independent family modulo $\de_0$. [This is simply a coding trick which allows
us to use a single function $g_*$ in the statement of the Theorem.]

\br
\step{3. Construction of $\de$.} 
By induction on $\alpha < \mu^+$ we choose a continuous sequence of triples
$(I, \de_{\alpha}, \gee_{\alpha})$ so that:
\begin{enumerate}
\item $(I, \de_{\alpha}, \gee_{\alpha})$ is $(\lambda, \leq \mu)$-good
\item $\beta < \alpha \implies \de_{\beta} \subseteq \de_{\alpha}$
\item $\alpha$ limit implies $\de_{\alpha}$ = $\bigcup \{ \de_{\beta} : \beta < \alpha \}$
\item $\gee_{\alpha} = \gee \cup \{ f^*_\epsilon : \epsilon \in [2\alpha, \mu^+) \}$
\item $\alpha$ limit implies $\gee_{\alpha}$ = $\bigcap \{ \gee_{\beta} : \beta < \alpha \}$ (follows)
\item if $\alpha = 0$ then $\de_{\alpha}$ is $\de_0$ 
\item if $\alpha = \beta + 1$ and $g \in X_\beta$ (see Step 0), then $B_{g,\beta} \in \de_\alpha$ where:
\[ B_{g,\beta} = \{ g(t) R^{M} f^*_{2\beta}(t) \equiv g(t) R^{M} f^*_{2\beta+1}(t) ~:
~\mbox{$t \in I$ and $f^*_{2\beta}(t), f^*_{2\beta+1}(t) \in M_0$ } \} \]
\item if $\alpha = \beta + 1$, then the set
\[  N_\beta = \{ t ~:~ f^*_{2\beta}(t) \neq f^*_{2\beta+1}(t) \} \]
belongs to $\de_{\alpha}$, and also for each constant function $c: I \rightarrow \mu$ 
the set
\[ S_\beta = \{ t ~: f^*_{2\beta}(t) > c(t) \} \cap \{ t ~: f^*_{2\beta + 1}(t) > c(t) \} \]
belongs to $\de_{\alpha}$.
\end{enumerate}

\br\noindent\emph{The induction}.
For $\alpha = 0$, let $\de_0, \gee_0$ be as defined above. 

For $\alpha$ limit, use (3) and (5), completing to a good triple, if necessary.

For $\alpha = \beta + 1$, consider the filter $\de^\prime_\beta$ generated by 
\[ \de_\beta \cup \{ B_{g,\beta} : g \in X_\beta \} \cup N_\beta \cup S_\beta \] 
Claim \ref{c:ind1} ensures that $(I, \de^\prime_\beta, \gee_\beta)$ is $(\lambda, \leq \mu)$-pre-good.
Choose $\de_\beta$ to be any filter extending $\de^\prime_\beta$ so that the triple
 $(I, \de_\beta, \gee_\beta)$ is $(\lambda, \leq \mu)$-good. This completes the inductive step. 

Finally, let $\de = \de_{\mu^+}$, and by construction $\gee_{\mu^+} = \gee$.
As explained at the beginning of the section, 
it follows from the cofinality of the construction that $(I, \de, \gee)$ is a good triple.

\br

\step{4. Distinct parameters.}
Here we justify the fact that the elements $\{ f^*_\gamma/\de_* : \gamma < \mu^+ \}$ are distinct in any ultrapower 
$M^\lambda/\de_*$ where $\de_* \supseteq \de$. It suffices to show this for any pair 
$f^*_\gamma, f^*_\zeta$.

If $\gamma = 2\beta, \zeta = 2\beta + 1$ then this is built in by Step 3, item 8. 

Otherwise, $\gamma, \zeta$ were dealt with at different stages and so will be distinct by Fact \ref{u-fact} 
and Step 3, item 9.

\br
\step{5. An omitted type.} 
In this step we prove that if $\de_* \supseteq \de$ is an ultrafilter on $I$ then ${M}^I/\de_*$ omits the type
\[ q(x) = \{ xR (f^*_\gamma/\de_*)^{\mbox{if ~($\gamma$ is even)}} ~: ~\gamma < \mu^+ \}  \]
By Step 4, the set of parameters is distinct, so $q$ is a consistent partial type.
Suppose for a contradiction that $\hat{g}  \in {^I M}$ realizes $q$. 
As observed at the beginning of the proof, 
since the cofinality of the construction is large ($\mu^+ > \mu$) we have for free that 
$(I, \de, \gee)$ is a good triple. Moreover, as $|M| = \mu < \mu^+$,  
for each element of the reduced product (i.e. each function $g: I \rightarrow M$) there is some $\beta = \beta_g < \mu^+$
such that for each $a \in M$, $A_{g,a}$ belongs already in $\mcp(I)/\de_\beta$ 
to the minimal completion of the subalgebra generated by 
$\langle f^{-1}(\{\alpha\}) : f \in \gee, \alpha < \mu \} \rangle$. 
Let $\beta = \beta_{\hat{g}}$.  By Step 3 item 7, 
\[ \{ t \in I : \hat{g} R f^*_{2\beta} ~\iff~ \hat{g} R f^*_{2\beta+1}  \} \in \de_{\beta+1} \subseteq \de \]
which gives the contradiction.
\end{proof}

We now give Fact \ref{u-fact} and Claim \ref{c:ind1} which were used in the construction. 
Fact \ref{u-fact} will ensure elements of the distinguished sequence built in \ref{claim-new} are distinct. 
In the language of order rather than equality, it is \cite{Sh:c} VI.3.19(1) p. 362.

\begin{fact} \label{u-fact} 
Let $\gee$ be independent $\mod \de$, and $\langle g^{-1}(t)/\de : t \in \mu \rangle$
a partition of $B(\de)$ for every $g \in \gee$ $($which holds if $(I, \de, \gee)$ is a good triple$)$. 
Suppose that $g^\prime: I \rightarrow \mu$, and $\mu^I/\de \models \epsilon \neq g^\prime/\de$ for every $\epsilon < \mu$,
and $g \in \gee$. 
Then for every ultrafilter $\de_* \supseteq \de$, $\mu^I/\de_* \models g/\de_* \neq g^\prime/\de_*$.
\end{fact}

\begin{proof}
Suppose to the contrary that
\[ X = \{ t \in I ~:~ g(t) = g^\prime(t) \} \neq \emptyset \mod \de \]
As $\{ g^{-1}(\epsilon) : \epsilon < \mu \}$ is a partition of $I$, there would have to be $\epsilon_* < \mu$
such that
$X \cap g^{-1}(\epsilon_*) \neq \emptyset \mod \de$. Thus
$\{ t \in I : g^\prime(t) = \epsilon_* \} \neq \emptyset \mod \de$, contradiction.
\end{proof}

Finally, Claim \ref{c:ind1} willl suffice to show that the
filter built in Step 3 does not contain $\emptyset$.
It says, roughly speaking, that if we are given a finite sequence $g_0, \dots g_{n-1}$ of elements of 
${^IM}$ such that the interaction of each $g_\ell$ with $M$ is supported by $\fin_s(\gee)$ in the sense described, and if
$\{f_1, f_2 \} \cup \gee$ form an independent family, then we may extend $D$ to ensure that $f_1, f_2$
are distinct nonstandard elements which nevertheless look alike to $g_0, \dots g_{n-1}$.  

\begin{claim} \label{c:ind1}
Suppose $(I, D, \gee \cup \{ f_1, f_2 \})$ is a $(\lambda, \mu)$-pre-good triple, $\rn(f_1) = \mu = \rn(f_2)$. 
Let $M$ be a base model of size $\mu$,
$|M| = \{ a_i : i < \mu \}$. 
Fix $n<\omega$ and let $g_0, \dots g_{n-1}: I \rightarrow M$ be functions such that for every $a \in M$ and $\ell<n$ 
the set $\{ t \in I ~:~ g_\ell(t) R^M a \}$ is supported by
$\fin_s(\gee)$ mod $D$. Then for any $A_h \in \fin_s(\gee)$ and any $\sigma \in [\mu]^{<\aleph_0}$ the set
\begin{align*} 
X :=  A_h ~\cap~ & \{ t \in I ~:~ \bigwedge_{\ell < n} g_\ell(t) R^M f_1(t) \equiv  g_\ell(t) R^M f_2(t) \}~~ \cap  \\
 & \{ t \in I ~:~ f_1(t) \neq f_2(t) \land ~( \{ f_1(t), f_2(t) \} \cap \{ a_i : i \in \sigma \} = \emptyset ) \} \\
& \neq \emptyset \mod D \\
\end{align*}
\end{claim}

\begin{proof}
In any ultrapower of the random graph, the assertion
\begin{align*} 
( & \forall y_0 \dots y_{n-1} ) (\forall z_0, \dots z_j) (\exists w_1 \neq w_2)(\exists \trv_0, \dots \trv_{n-1} \in \{ 0,1\} ) \\
& \left(   ( \{ z_0, \dots z_j \} \cap \{ w_1, w_2 \} = \emptyset) \land \bigwedge_{\ell < n} (y_i R w_1) \equiv (y_i R w_2)  \equiv (\trv_\ell=1) \right) \\
\end{align*}
will be true by \lost theorem. Thus, since $A_h \neq \emptyset \mod \de$, 
it must be consistent with the nontrivial filter generated by $\de \cup \{ A_h \}$ to choose $a, a^\prime \in M$, $\trv_\ell \in \{ 0, 1\}$ 
such that $a, a^\prime$  play the role of $w_1, w_2$ when we replace the $y$s by $g$s and the $z$s by the elements $\{ a_i : i \in \sigma \}$ of $M$. 

Now we assumed in the statement of the claim that each of the sets
\[ (g_\ell R a)^{\trv_\ell}, (g_\ell R a^\prime)^{\trv_\ell}, \ell < n \]
are supported by $\fin_s(\gee) \mod \de$. Moreover, as just shown, their intersection is nonempty $\mod \langle \de \cup A_h \rangle$.
So we may find some nontrivial $A_{h^\prime} \in \fin_s(\gee)$, with $h \subseteq h^\prime$,
contained in that intersection $\mod \de$. By the assumption that $\{ f_1 ,f_2 \} \cup \gee$ is independent $\mod \de$, we have that
\[ A_{h^\prime} \cap A_h \cap  f^{-1}_1(a) \cap f^{-1}_2(a^\prime) \neq \emptyset \mod \de \]
As this set is clearly contained in the set $X$ from the statement of the Claim, we finish the proof. 
\end{proof}

\begin{rmk}
Much about this argument is not specific to the theory of the random graph. 
\end{rmk}

\begin{disc}
In \cite{MiSh:996} we gave a result showing ``decay of saturation'' for non-simple theories using a combinatorial principle from \cite{KSV}. 
It would be very useful if the construction just given could be extended to e.g. the Keisler-minimum $TP_2$ theory, known from 
Malliaris \cite{mm4} as the theory $T_{feq}$ of infinitely many equivalence relations each
with infinitely many infinite classes which generically intersect. The analysis of ``fundamental formulas'' in \cite{mm5} may be relevant. 
\end{disc}

We conclude this section by showing:

\begin{cor} \label{cor1}
If there exists $\mu$ s.t. $\mu^{++} \leq \lambda$, i.e. if $\lambda \geq \aleph_2$, then there is a regular ultrafilter $\de_*$ on $\lambda$ such that
$\lcf(\aleph_0, \de_*) \geq \lambda^+$ but $\de_*$ does not saturate the random graph. 
\end{cor}

\begin{proof}
First, apply Theorem \ref{claim-new} in the case $\mu = \aleph_0$ 
to build $\de$ as stated there, so $(I, \de, \gee)$ is $(\lambda, \aleph_0)$-good. By Fact \ref{fact-lcf} and the assumption that
$|\gee| = 2^\lambda$, there is no barrier to constructing $\de_* \supseteq \de$ so that $\lcf(\aleph_0, \de_*) \geq \lambda^+$.
\end{proof}

\begin{rmk}
Corollary \ref{cor1} gives a proof in ZFC of a result which we had previously shown using the existence of a measurable cardinal, 
see the table of implications in \cite{MiSh:996} \S 4. 
\end{rmk}

\section{For $\kappa$ measurable, $\lambda \geq \kappa^+$ there is $\de$ on $\lambda$ flexible, not good} \label{s:meas}
In this section we prove that for $\lambda \geq \kappa^+$ and $\kappa > \aleph_0$ a measurable cardinal, there is a regular
ultrafilter $\de$ on $\lambda$ which is $\lambda$-flexible but not $\kappa^{++}$-good. 
This addresses a problem from Dow 1975 [asking whether, in our language, $\lambda^+$-flexible implies $\lambda^+$-good]. 
This complements our answer in the companion paper, \cite{MiSh:996} Theorem 6.4,
which showed by taking a product of a regular, good ultrafilter on $\lambda = \lambda^\kappa$ with a $\kappa$-complete ultrafilter on $\kappa$
that there is a regular $\de$ on $\lambda$ which is flexible but not $(2^\kappa)^+$-good for any unstable theory. 
There are some analogies between the proofs, but the method developed here appears more general.
We improve the gap between the cardinals, however we can say less about the source of the omitted type.

Our setup for this section will be as follows. We start with $\ba$, a complete $\kappa^+$-c.c. Boolean algebra, $\kappa \leq \lambda$.
By transfinite induction we build a filter $\de$ on an index set $I$, $|I| = \lambda$ so that there is a surjective homomorphism
$h : \mcp(I) \rightarrow \ba$ with $h^{-1}(1_\ba) = \de$. We can then extend $\de$ to an ultrafilter by choosing some 
ultrafilter $E$ on $\ba$ and letting $\fil = \{ A \subseteq I : h(A) \in E \}$. In the main case of interest, $E$ is $\theta$-complete
for some $\theta$, $\aleph_0 < \theta \leq \kappa$ (necessarily measurable). Let us now give such objects a name. 

\begin{defn} \label{p21}
Let $\mk$ be the class of $\mx = ( I, \lambda, \kappa, \theta, \de, \mh, \ba, E)$ where $I$ is a set of cardinality $\geq \lambda$, and:
\begin{enumerate}
\item $\de$ is a $\lambda$-regular filter on $I$
\item $\ba$ is a complete $\kappa^+$-c.c. Boolean algebra, $\kappa \leq \lambda$.
\item $\mh$ is a homomorphism from $\mathcal{P}(I)$ onto $\ba$ with $\de = \mh^{-1}(1_{\ba})$
\item $E$ is an (at least) $\theta$-complete ultrafilter on $\ba$
\item if $\epsilon_* < \theta$ and $\langle a_\epsilon : \epsilon < \epsilon_*\rangle$ is a maximal antichain of $\ba$ then we can find a partition
$\langle A_\epsilon : \epsilon < \epsilon_* \rangle$ of $I$ such that
$\bigwedge_\epsilon \mh(A_\epsilon) = a_\epsilon$
\end{enumerate}
\end{defn}

Recall that $\langle a_\epsilon : \epsilon < \epsilon_* \rangle$ is a maximal antichain of $\ba$ when 
each $a_\epsilon > 0_\ba$, and $\epsilon < \zeta \implies a_\epsilon \cap a_\zeta = 0_\ba$.
We will use $\de_\bx, I_\bx$, and so on to refer to objects from the tuple $\bx$.

\begin{defn}
For $\bx \in \mk$, we define:
\begin{enumerate}
\item $\fil_\bx = \{ A \subseteq I : h(A) \in E_\bx \}$,  the ultrafilter induced on $I$ 
\item Let $\Theta_\bx = $
\[ \{ \theta_* ~:~ \mbox{there is a partition $\langle A_i : i < \theta_* \rangle$ of $I$ s.t. 
$\langle \mh(A_i) : i < \theta_* \rangle$ is a maximal antichain of $\ba$} \} \]
\item We say $\bx$ is $\sigma$-good if $\de_\bx$ is $\sigma$-good; we say it is good if $\de_\bx$ is $\lambda^+$-good
\item We say $\bx$ is flexible when $\de_\bx$ is $\lambda$-flexible.
\end{enumerate}
\end{defn}

\begin{rmk}
Note that in $\bx$, Definition \ref{p21}(5), we distinguish one lower-case $\theta$ with a related but not identical meaning:
if $\theta_\bx = \aleph_1$, then $\aleph_0 \in \Theta_\bx$.
\end{rmk}

In the next definition, we ask about the reverse cofinality of certain sets of $\fil(\mx)$-nonstandard elements
which are already $\de_\mx$-nonstandard. 
 
\begin{defn} \emph{(On coinitiality)} \label{lcf-sigma}
\begin{enumerate}
\item For $\bx \in \mk$ let  
\[ \eff_{\sigma, \de_\bx} = \{ f \in {^I\sigma} ~:~ ({i<\sigma} \implies i<_{\fil(\bx)} f ) ~~\mbox{moreover}~~ 
({i<\sigma} \implies \{ t \in I : f(t) > i \} \in \de_\bx) \} \]
\item If there exists $\delta = \cf(\delta)$ and $f_\alpha \in \eff_{\sigma, \de_{\mx}}$
for $\alpha < \delta$ such that:
\begin{itemize}
\item $\alpha < \beta  \implies f_\beta <_{\de_\mx} f_\alpha$
\item if $f \in \eff_{\sigma, \de_{\mx}}$ then $\bigvee_{\alpha < \delta} f_\alpha <_{\de_\mx} f$
\end{itemize} 
then say that $\lcf_\sigma(\bx) := \delta$. Here $\lcf_\sigma(\bx)$ is the true cofinality of $(\eff_{\sigma, \de_\bx}, >_{\de_\bx})$,
which is not always well defined but is equal to $\delta$ if such exists. $($This is not the same as 
$\lcf_\sigma(\fil_\mx) = \delta$.$)$ Below, ``assume $\lcf_\sigma(\bx)$ is well defined''
will mean: suppose there is such a $\delta$.

\item For $F$ an ultrafilter on $I$, $|I| = \lambda$, and $\aleph_0 \leq \sigma \leq \lambda$ the lower cofinality of $\sigma$ with respect to $F$,
$\lcf(\sigma, F)$ is the cofinality of the set of elements above the diagonal embedding of $(\sigma, <)$ in $(\sigma, <)^\lambda/F$ considered with the
reverse order, i.e. the coinitiality of $\sigma$ in $(\sigma, <)^\lambda/F$.
\end{enumerate}
\end{defn}

In the remainder of the section, we work towards a proof of Theorem \ref{flex-not-good-b}.

First, in Claim \ref{bx-exists} we show existence of a $\lambda^+$-good, $\lambda$-regular filter $\de$ with the
desired map to $\ba$. 

\begin{claim} \label{bx-exists}
Assume $\kappa > \aleph_0$ is a measurable cardinal, $E$ a uniform $\kappa$-complete ultrafilter on $\kappa$, $\lambda \geq \kappa$,
$\delta = \cf(\delta) \in [\lambda^+, 2^\lambda]$. 
Then there is $\bx \in \mk$ such that:
\begin{enumerate}
\item $\lambda_\bx = I_\bx = \lambda$
\item $\ba$ is the Boolean algebra $\mcp(\kappa)$, $\kappa_\mx = \kappa$, \item $E_\mx = E$, $\theta_\mx = \aleph_1$
\item $\de_\mx$ is a $\lambda^+$-good, $\lambda$-regular filter on $I$
\item $\kappa \in \Theta_\mx$
\item (if desired) $\lcf(\mx) = \delta$  (remark: can allow smaller $\delta$ but then goodness goes down)
\end{enumerate}
\end{claim}

\begin{proof} 
There are three main steps.

\step{Step 0: Setup.}
We begin with a $(\lambda, \lambda)^+$-good triple $(I, \de_0, \gee_*)$ [here the $^+$ means all functions have range $\lambda$ as
opposed to $\leq \lambda$]. 
That is: $\gee_* \subseteq {^I\lambda}$ is an independent family of functions
each of which has range $\lambda$, $|\gee_*| = 2^\lambda$, $\de_0$ is a $\lambda$-regular filter on the index set $I$, and $\de_0$ is maximal such that 
the family $\gee$ is independent modulo $\de_0$. 

$\ba = \mcp(\kappa)$ is a complete $\kappa^+$-c.c. Boolean algebra, so condition (5) will follow from the fact that $\mh$ is a homomorphism with range $\ba$. 

If we want to ensure condition (6), we may for transparency begin with a good triple $(I, \de_0, \gee_* \cup \eff_*)$ where 
$\gee_* \subseteq {^\lambda \lambda}$, $\eff_* \subseteq {^\lambda \aleph_0}$ and $|\gee_*| = |\eff_*| = 2^\lambda$. We then modify the construction below by enumerating the steps 
by an ordinal divisible by $2^\lambda$ and with cofinality $\delta$, e.g. $2^\lambda \times \delta$. 
Enumerate $\eff_*$ as $\langle f_\alpha : \alpha < 2^\lambda \times \delta \rangle$. 
At odd steps in the induction we proceed as stated below, and at even steps $2\alpha$ we consume the function $f_\alpha$. Since 
$\{ f_\gamma : \gamma > \alpha \}$ will remain independent modulo $\de_\alpha$ under this modification, the sequence
$\langle f_{2^\lambda \times i} : i < \delta \rangle$ will witness the desired lower cofinality. Note that the assumption $\delta \geq \lambda^+$
is important here as we may not otherwise obtain an ultrafilter.

\step{Step 1: Setting up the homomorphism.}

Let $\gee_* = \eff \cup \gee$, $\eff \cap \gee = \emptyset$, $|\eff| = 2^\lambda$, and $\gee = |\ba|$.

We first define several subalgebras of $\ba$: 

Let $\ma_2 = \{ A \subseteq I : A  ~\mbox{is supported by}~ \gee \mod \de_0\}$. Recall that $A \subseteq I$ is supported by 
an independent family mod $\de_0$ if there is a partition $\{ X_j : j \in J \}$ of the index set $I$ with $j \in J \implies
X_j \in \fin_s(\gee)$ and such that for each $j \in J$, either $X_j \subseteq A \mod \de_0$ or $X_j \cap A = \emptyset \mod \de_0$. 
Thus $\ma_2$ is a subalgebra of $\mcp(I)$ which contains all $X \in \de_0$ as well as $\{ g^{-1}(0) : g \in \gee \}$.  

Let $\langle g_{a} : a \in \ba \rangle$ enumerate $\gee$. 

Let $\ma_0$ be the subalgebra $\{ A \subseteq I : A \in \de_0~\mbox{or} ~I \setminus A \in \de_0\}$.

Let $\ma_1$ be the subalgebra generated by $\ma_0 \cup \{ g^{-1}_b(0) : b \in \ba \}$. 

We now define several corresponding homomorphisms:

Let $h_0$ be the homomorphism from $\ma_0$ into $\ba$ given by $A \in \de_0\implies h_0(A) = 1_\ba$.

Let $h_1$ be a homomorphism from $\ma_1$ into $\ba$ which extends $h_0$ such that $h_1(g^{-1}_{a}(0)) = a$. 

Let $h_2$ be a homomorphism from $\ma_2$ into $\ba$ which extends $h_0$; this is possible since $\ba$ is complete. 

\step{Step 2: The inductive construction.}
Having defined $h_2$, we would like to extend $\de_0$ to a filter $\de$  and $h_2$ to $h$ as given by the Claim, and we have two 
main tasks to accomplish: first, that $\de$ is $\lambda^+$-good, and second that  $h: \mcp(I) \rightarrow \ba$ satisfies
$A \in \de \implies h(A) = 1_\ba$.

Let $\langle A_i : i < 2^\lambda \rangle$ enumerate $\mcp(I)$. Let $\langle \overline{A}^\times_j : j < 2^\lambda \rangle$ enumerate
all possible multiplicative tasks, each occurring cofinally often. 
Recall that in order to ensure our eventual filter $\de$ is good we need to ensure that any monotonic function
$\pi: \fss(\lambda) \rightarrow \mcp(I)$ whose range is included in $\de$ has a refinement $f^\prime$ which is multiplicative, i.e.
$\pi(u) \cap \pi(v) = \pi(u \cup v)$. To produce the list $\langle \overline{A}^\times_j : j < 2^\lambda \rangle$, where each $\overline{A}^\times_j$ 
is a sequence of elements of $\mcp(I)$, fix in advance some enumeration of $\fss(\lambda)$ and identify each monotonic $\pi: \fss(\lambda) \rightarrow
\mcp(I)$ with its range. 

Interpolating these two lists gives a master list of tasks for the construction: at odd successor stages we will deal with sets $A_i \subseteq I$,
at even successor stages we deal with possible multiplicative tasks $\overline{A}^\times_j$.
We choose $(D_\alpha, h_\alpha, \ma_\alpha, \eff_\alpha)$ by induction on $\alpha \in [2, 2^\lambda)$ such that:

\begin{enumerate}
\item $\de_2 = \{ A \subseteq I : h_2(A) = 1 \}$
\item $\langle \de_\alpha : \alpha \in [2, 2^\lambda] \rangle$ is increasing continuously with $\alpha$, and $\de_{2^\lambda}$ is an ultrafilter
\item $\langle \eff_\alpha : \alpha \in [2, 2^\lambda] \rangle$ is decreasing continuously with $\alpha$; $\eff_2 = \eff$, $\alpha < \beta \implies \eff_\alpha \supseteq \eff_\beta$, $|\eff_\alpha \setminus \eff_{\alpha + 1} | \leq \lambda$, $\eff_{2^\lambda} = \emptyset$
\item $(I, \de_\alpha, \eff_\alpha \cup \gee)$ is a good triple, meaning that the functions $\eff_\alpha \cup \gee$ 
remain independent modulo $\de_\alpha$ and $\de_\alpha$ is maximal for this property
\item $\ma_\alpha = \{ A \subseteq I ~:~ \mbox{there is $B \in \dom(h_2)$ such that} ~A = B \mod \de_\alpha \}$ 
\item $h_\alpha \in \Hom(\ma_\alpha, \ba)$ and $\langle h_\alpha : \alpha \in [2, 2^\lambda] \rangle$ is increasing continuously with $\alpha$
\item if $\alpha = 2\beta + 1$, then $A_\beta \in \ma_\alpha$
\item if $\alpha = 2\beta + 2$ and $\overline{A}^\times_\alpha = \langle A^\times_{\beta, u} : u \in [\lambda]^{<\aleph_0} \rangle$,
and $u \in [\lambda]^{<\aleph_0} \implies A^\times_{\beta, u} \in D_{2\beta + 1}$, then it has a multiplicative refinement in 
$D_\alpha$  
\end{enumerate}

For condition (7), $\alpha = \beta + 1$: if neither $A$ nor $I \setminus A$ is empty modulo $\de_\beta$ (in which case we are done) apply Claim \ref{hom-base}
which returns a pre-good triple $(I, \de^\prime, \gee \cup (\eff_\beta \setminus \eff^\prime))$, where $|\eff^\prime| \leq \lambda$.
Let $\eff_{\alpha} = \eff_\beta \setminus \eff^\prime$, and
without loss of generality, extend the pre-good triple returned by Claim \ref{hom-base} to a good triple 
$(I, \de_\alpha, \gee \cup \eff_\alpha)$.

For condition (8), $\alpha = \beta + 2$: choose any $g \in \eff_{\beta + 1}$ and let $\eff_\alpha = \eff_{\beta+1} \setminus \{ g \}$. Using $g$, we can produce
a multiplicative refinement for $\langle A^\times_{\alpha, u} : u \in [\lambda]^{<\aleph_0} \rangle$ while keeping $\eff_{\beta+1} \setminus \{ g \}$
independent, as in the usual construction of good filters.
See for instance \cite{Sh:c} Claim 3.4 p. 346.

Finally, we verify that the construction satisfies Definition \ref{p21}(6) for $\theta_\bx = \aleph_1$:  
As $\epsilon_* < \theta = \aleph_1$, without loss of generality $\epsilon_* \leq \omega$ so without loss of generality $\epsilon_* = \omega$. 
We can choose by induction on $n$, sets $A^\prime_n \subseteq I \setminus \bigcup \{ A^\prime_m : m < n \}$ such that $\mh(A^\prime_n) = a_n$,
and then let $A_n$ be $A^\prime_n$ if $n> 0$, and $I \setminus \{ A^\prime_{1+k} : k < \omega \}$ if $n=0$.
\end{proof}

We now prove a technical claim for Claim \ref{bx-exists}, used to ensure (7) of the induction. 

For ``supported,'' recall Definition \ref{d:ba-various} above. Note that in Claim \ref{hom-base}, what is shown is that $X$ 
is equivalent modulo $\de^\prime$ to a set elements of $\fin_s(\gee)$ which are pairwise disjoint (modulo $\de$). 
If desired, repeat the proof for $I \setminus X$ in place of $X$ to explicitly obtain a partition. 

\begin{claim} \label{hom-base}
Let $(I, \de, \gee \cup \eff)$ be a $(\lambda, \lambda)$-good triple. Let $X \subseteq I$, $X \in \de^+$.
Then there are $\eff^\prime \subseteq \eff, |\eff^\prime| \leq \lambda$ and a filter $\de^\prime \supseteq \de$
such that $(I, \de^\prime, \gee \cup (\eff \setminus \eff^\prime))$ is a $(\lambda, \lambda)$-good triple, and moreover
$X$ is supported by $\gee \mod \de^\prime$.
\end{claim}

\begin{proof}
Our strategy is to try to choose $A_{h_i} \in \fin_s(\gee \cup \eff)$, $\de_i \supseteq \de$, $\eff^i \subseteq \eff$ 
by induction on $i < \lambda^+$, subject to the following conditions. 
\begin{enumerate}
\item For each $i$, $A_{h_i} \neq \emptyset \mod \de_i$, and either $A_{h_i} \subseteq X \mod \de_i$ or else $A_{h_i} \cap X = \emptyset \mod \de_i$
\item $j<i$ implies $A_{h_j} \cap A_{h_i} = \emptyset \mod \de_i$, but
\[   f \in \left(  \dom(h_j) \cap \eff  \right) \cap \left(  \dom(h_i) \cap \eff  \right) 
\implies h_j(f) = h_i(f)    \]
\item For each $i$, $\eff^i = \{ f \in \eff : (\forall j < i) (f \notin \dom(h_j) \}$
\item For each $i$, $\de_{i}$ extends the filter generated by 
\[ \de \cup \{ \{ s \in I : f(s) = h_j(f) \} : j < i, f \in \dom(h_j) \cap \eff\} \]
and moreover $(I, \de_i, \gee \cup \eff^i)$ is a $(\lambda, \lambda)$-good triple.
\end{enumerate}

For $i=0$, let $\de_0 = \de, \eff^0 = \eff$, and choose $A_{h_0} \subseteq X \mod \de$ by Fact \ref{good-dense}.

For $i>0$, we first describe the choice of $\de_i$. Let $\de^*_{i}$ be the filter generated by 
\[ \de \cup \{ \{ s \in I : f(s) = h_j(f) \} : j < i, f \in \dom(h_j) \cap \eff\} \]
The choice of $\eff^i$ is determined by condition (3). 
Notice that by condition (2) above, it will always be the case that
\[ (I, \de^*_i, \gee \cup \eff^i) \]
is a $(\lambda, \lambda)$-pre-good triple. Without loss of generality, we may extend $\de^*_i$ to $\de_i$ which satisfies
condition (4). 

\br
For the choice of $A_{h_i}$, the limit and successor stages are the same, as we now describe.

At stage $i < \lambda^+$, suppose we have defined $\de_i$. 

Let $Y_i = \bigcup \{ A_{h_j} : j < i , A_{h_j} \subseteq X \mod \de_i \}$
and let $\eff^i = \{ f \in \eff : (\forall j < i) (f \notin \dom(h_j) \}$. 

When attempting to choose $A_{h_i}$, one of two things may happen.

\step{Case 1.} 
The ``remainder'' is already small, meaning that 
\[ I \setminus Y_i = \emptyset \mod \de_i \] 
In this case, the internal approximation to $X$ indeed works, i.e.
\[ X = \bigcup\{ A_{h_j} : j < i, A_{h_j} \subseteq X \mod \de_i \}  \mod \de_i \]
For each $j < i$ let $h^\prime_j$ be the restriction of $h_j$ to $\gee$. Then from the definition
of $\de_i$, 
\[ A_{h_j} \subseteq X \mod \de_i ~~ \iff  A_{h^\prime_j} \subseteq X \mod \de_i  \]
Moreover for $j<k<i$, by condition (2) above, this remains a partition:
\[ A_{h_j} \cap A_{h_k} =\emptyset \mod \de_k \implies A_{h^\prime_j} \cap A_{h^\prime_k} =\emptyset \mod \de_i \]
Thus $X$ is supported by $\fin_s(\gee) \mod \de_i$, namely
\[ X = \bigcup\{ A_{h^\prime_j} : j < i, A_{h_j} \subseteq X \mod \de_i \}  \mod \de_i \]
and we set $\de^\prime = \de_i$ and $\eff^\prime = \eff \setminus \eff^i$ to finish the proof. 

\step{Case 2.} Not Case 1, in which case
\[ I \setminus Y_i \neq \emptyset \mod \de_i \]
so as $(I, \de_i, \gee \cup \eff^i)$ is a good triple we may choose $A_{h_i} \subseteq I \setminus Y_i \mod \de$
with $h_i \in \fin_s(\gee \cup \eff^i)$, by Fact \ref{good-dense}. Note that compliance with (2) is ensured by
the definition of $\de_i$. This completes the inductive step at $i$. 

\br
Finally, suppose for a contradiction that the induction continues for all $i<\lambda^+$. Then by construction, 
$(I, \de_{\lambda^+}, \eff^{\lambda^+} \cup \gee)$ is a $(\lambda, \lambda)$-good triple,
and by Fact \ref{mcc}, $CC(B(\de_{\lambda^+})) \leq \lambda^+$. However, by the ``moreover'' line in Case 1,
$\{ A_{h_i} : i < \lambda^+ \}$ is a partition of $B(\de)$, contradiction. 

Thus the induction stops (i.e. we reach Case 1) at some bounded stage $i_* < \lambda^+$.
In particular, it will be the case that $|\eff^\prime| \leq \lambda$.
This completes the proof. 
\end{proof}

We now bring in $E$, and verify that the induced filter $\fil_{\mx}$ is regular and that certain $\fil(\bx)$-nonstandard elements
are $\fil(\bx)$-equivalent to elements which are already $\de_\bx$-nonstandard. 

\begin{claim} \label{p29a}
Assume $\bx \in \mk$ and $\sigma = \cf(\sigma) \in \Theta_\bx$.
\lp e.g. $\sigma = \aleph_0$ for the $\bx$ from Claim \ref{bx-exists}\rp
\begin{enumerate}
\item $\fil(\bx)$ is a regular ultrafilter on $I$
\item If $\sigma = \cf(\sigma) < \theta$ 
\newline \lp or just $\sigma = \cf(\sigma) \in \Theta_\bx$ : in our case, $\sigma = \aleph_0$ \rp

\noindent then if $g \in {^I \sigma}$ and $\bigwedge_n n <_{\fil(\bx)} g$ there is $f \in {^I \sigma}$ such that:
\begin{enumerate}
\item $f = g \mod \fil(\bx)$
\item $i < f \mod \de_\bx$ for every $i < \sigma$
\end{enumerate}
\end{enumerate}
\end{claim}

\begin{proof}
(1) $\fil(\bx)$ is an ultrafilter by Definition \ref{p21}(3)-(4), and inherits regularity from $\de$.

(2) Let $\sigma = \aleph_0$, $g \in {^I\sigma}$ be given, and suppose that $n \in \mathbb{N} \implies n <_{\fil(\bx)} g$.
For each $\epsilon < \sigma$, define $A_\epsilon = \{ t \in I : g(t) < \epsilon \}$, so 
$\langle A_\epsilon : \epsilon < \sigma \rangle$ is a $\subseteq$-increasing sequence of subsets of $I$ whose union is $I$. 
Let $a_\epsilon = \mh(A_\epsilon)$. 

If for some $\epsilon < \sigma$ we have that $a_\epsilon \in E$, then by definition $A_\epsilon \in \fil(\mx)$, and thus
$\epsilon \geq g \mod \fil(\mx)$, contradicting our assumption on $g$. So for all $\epsilon < \sigma$, $a_\epsilon \notin E$.
Since $E$ is a $\theta$-complete ultrafilter on $\ba$, there is $a \in E$ such that for all $\epsilon < \sigma$,
$0 <_\ba a \leq_\ba (1_\ba- a_\epsilon)$. 
Let $A \subseteq I$ be such that $\mh_\bx(A) = a$, so $\epsilon < \sigma \implies A \cap A_\epsilon = \emptyset \mod \fil(\bx)$. 

Since $\de_\bx$ is $\lambda$-regular, there is $g_1 \in {^I \sigma}$ such that $i < \sigma \implies i < g_1 \mod \de_\bx$ 
(it suffices to majorize some finite set of elements at each index). Now we can define $f \in {^I\sigma}$ by:
\[ f(t) = \begin{cases}
g(t) ~~ \mbox{if $t \in A $}  \\
g_1(t) ~~\mbox{if $t \in I \setminus A $} \\
\end{cases} \]
Clearly $f$ satisfies our requirements. 
\end{proof}

In Claim \ref{p29b} we verify that flexibility, which we were guaranteed only for $\de$, in fact remains true for the 
induced ultrafilter $\fil_\mx$.

\begin{claim} \label{p29b}
\begin{enumerate}
\item Assume $\aleph_0 < \theta_\bx$ (as holds in Claim \ref{bx-exists}). If $\de_\bx$ is a $\lambda$-flexible filter,
as in Claim \ref{bx-exists}, then also $\fil(\bx)$ is a $\lambda$-flexible ultrafilter. 
\item Assume that $\aleph_0 < \theta_\bx$, and $\lcf(\bx)$ is well defined, see above. Then $\lcf(\aleph_0,\fil(\bx)) = \lcf_{\aleph_0}(\bx)$.
Similarly for $\lcf(\sigma, \fil(\bx))$ and $\lcf_\sigma(\de_\bx)$. 
\end{enumerate}
\end{claim}

\begin{proof}
(1) As in Definition \ref{flexible}, let $g \in {^I\mathbb{N}}$ be such that $n \in \mathbb{N} \implies n<g \mod \fil(\mx)$. 
We would like to find a sequence $\langle U_\alpha : \alpha < \lambda \rangle$ of members of $\fil(\mx)$ such that
$t \in I \implies g(t) \geq |\{ a< \lambda : t \in U_\alpha \}$. 

Recalling that $\aleph_0 < \theta$, let $f$ be the function corresponding to $g$ from Claim \ref{p29a}(2). 
Let $A = \{ t \in I : f(t) = g(t) \} \in \fil(\bx)$. 
By assumption, $\de_\bx$ is a $\lambda$-flexible filter, so there is a sequence $\langle U^\prime_\alpha : \alpha < \lambda \rangle$
of members of $\de_\bx$ such that $t \in I \implies (f(t) \geq | \{ a < \lambda : t \in U^\prime_\alpha \} |$. For each $\alpha < \lambda$,
define $U_\alpha = U^\prime_\alpha \cap A$, and $U_\alpha \in \fil(\bx)$ as filters are closed under intersection. 
Finally, we verify $\langle U_\alpha : \alpha < \lambda \rangle$ is the desired regularizing set, by cases. 
If $t \in A$, by choice of the $\langle U^\prime_\alpha : \alpha < \lambda$ we have that $f(t) \geq | \{ a < \lambda : t \in U^\prime_\alpha \} |$,
and if $t \notin A$, $|\{ \alpha : t \in U_\alpha \}| = 0 \leq g(t)$.

(2) Let $\sigma$ be given. Let $\eff_1 = \{ f \in {^I\sigma} : i <\sigma \implies i < f \mod \de_\bx \}$, and  let
$\eff_2 = \{ f \in {^I\sigma} : i <\sigma \implies i < f \mod \fil(\bx) \}$. 
Then $\eff_1 \subseteq \eff_2$, and if $f_1, f_2 \in \eff_1$ and $\eff_1 \leq \eff_2 \mod \de_\bx$ then $f_1 \leq f_2 \mod \fil(\bx)$. 

Let $\delta = \lcf_\sigma(\de_\bx)$ and let $\langle f_\alpha : \alpha < \delta \rangle$ witness it, as in Definition \ref{lcf-sigma}. 
We would like to show that this same sequence witnesses the coinitiality of $\sigma$ modulo $\fil(\bx)$, i.e. that the two conditions of
Definition \ref{lcf-sigma} are satisfied for $\fil(\bx)$ in place of $\de_\bx$. First, since 
$\de_\bx \subseteq \fil(\bx)$, $\alpha < \beta < \delta \implies f_\beta <_{\fil(\bx)} f_\alpha$. 
Second, for any $g \in \eff_2$, by Claim \ref{p29a}(2) there is some $f \in \eff_1$ such that $g = f \mod \fil(\bx)$,
and for some $\alpha < \delta$, $f_\alpha < f \mod \de_\bx$. Hence $f_\alpha \leq f=g \mod \fil(\bx)$. The sequence
$\langle f_\alpha/\fil(\bx) : \alpha < \delta \rangle$ cannot be eventually constant, as then $\lcf_\sigma(\fil(\bx) = 1$,
so we are done.
\end{proof}

\begin{defn} \emph{(Goodness for a boolean algebra)} \label{ba-goodness}
We say that the subset $X \subseteq \ba$, usually a filter or ultrafilter, 
is $\mu^+$-good if every monotonic function from $\fss(\mu) \rightarrow \ba$ has a multiplicative
refinement, i.e. for every sequence $\langle a_u : u \in [\mu]^{<\aleph_0} \rangle$ of members of $X$ with 
$u \subseteq v \implies a_v \subseteq a_u$, there is a refining sequence 
$\langle b_\alpha : \alpha \in \mu \rangle$ of members of $X$ such that for each $u \in [\mu]^{<\aleph_0}$, 
$\ba \models $ `` $(\bigcap_{\alpha \in u} b_\alpha)  \leq a_u $''. 
\end{defn}

\begin{claim} \label{p32}
Suppose $\bx \in \mk$ and $E$ is not $\rho^+$-good considered as a subset of $\ba$ \lp e.g. $\ba = \mcp(\kappa)$ and $\rho = \kappa^+$ so 
$\rho^+ = \kappa^{++}$\rp. Then $\fil(\bx)$ is not $\rho^+$-good.
\end{claim}

\begin{proof}
If $E$ is not $\rho^+$-good, then as in Definition \ref{ba-goodness} 
we can find a sequence $\langle a_u : u \in [\rho]^{<\aleph_0} \rangle$ of members of $E$ such that there is no 
refining sequence $\langle b_\alpha : \alpha < \rho \rangle$ of members of $E$ such that for each $ u \in [\rho]^{<\aleph_0}$,
$y_u \leq_\ba a_u$ where $y_u := \bigcap_{\alpha \in u} b_\alpha$. 
Let $A_u \subseteq I$ be such that $\mh_\bx(A_u) = a_u$. Choosing $A_u$ by induction on $|u|$, we can ensure monotonicity, i.e.
that $u \subseteq v \implies A_v \subseteq A_u$. Now if $\fil(\bx)$ were $\rho^+$-good, we could find a sequence $\langle B_\alpha : \alpha < \mu \rangle$
of members of $\fil(\bx)$ such that for each $u \in [\rho]^{<\aleph_0}$, $\bigcap_{\alpha \in u} B_\alpha \subseteq A_u$. But then 
letting $b_\alpha = \mh(B_\alpha)$ for $\alpha <\mu$ gives a refining sequence which 
contradicts the choice of $\langle a_u : u \in [\rho]^{<\aleph_0} \rangle$.
\end{proof}

\begin{theorem} \label{flex-not-good-b}
Assume $\kappa$ is a measurable cardinal, $\kappa < \lambda$.  There is $\de$ such that:
\begin{enumerate}
\item $\de$ is a regular ultrafilter on $\lambda$
\item $\de$ is flexible, i.e. $\lambda$-flexible
\item $\de$ is not $\kappa^{++}$-good
\end{enumerate}
\end{theorem}

\begin{proof}

Let $E$ be a uniform $\kappa$-complete ultrafilter on $\kappa$, which exists as $\kappa$ is measurable. Let $\ba = \mcp(\kappa)$.
Let $\bx$ be as in Claim \ref{bx-exists}.  
Then:
\begin{enumerate}
\item  $\fil(\bx)$ is an ultrafilter on $\lambda = I_\bx = \lambda_\bx$ as $\bx \in \mk$, by Claim \ref{p29a}.
Since $\de_\bx$ is a $\lambda^+$-good $\lambda$-regular filter on $\lambda$, we have that $\fil(\bx)$ is regular by
Claim \ref{bx-exists}(3).

\item $\fil(\bx)$ is $\lambda$-flexible by Claim \ref{p29b}.

\item $\fil(\bx)$ is not $\kappa^{++}$-good by Claim \ref{p32}.
\end{enumerate}
This completes the proof.
\end{proof}

\section{For $\kappa$ weakly compact $\de$ may have no $(\kappa, \kappa)$-cuts while $\lcf(\aleph_0, \de)$ is small} \label{s:weakly-compact}

In this section we prove, using a weakly compact cardinal $\kappa$, that it is possible to realize all $(\kappa, \kappa)$-pre-cuts while
allowing $\lcf(\aleph_0)$ to be small \emph{thus} failing to saturate any unstable theory. In some sense, we play weak compactness against
the cofinality of the construction. Throughout this section $\kappa$ is weakly compact,  
but we will mention where we use this. 

\begin{defn} \label{wcc} The cardinal $\kappa$ is weakly compact if for every $f: \kappa \times \kappa \rightarrow \{ 0, 1\}$ there is 
$\uu \subseteq \kappa$, $|\uu| = \kappa$ and $\trv \in \{ 0, 1\}$ such that for all $\epsilon < \zeta$ from $\uu$,
$f(\epsilon, \zeta) = \trv$.
\end{defn}

\begin{fact} \label{wcc-fact} \emph{(see e.g. Kanamori \cite{kanamori} Theorem 7.8 p. 76)}
If $\kappa > \aleph_0$ is weakly compact, $n < \aleph_0$ and $\rho < \kappa$, then for any $\alpha: [\kappa]^n \rightarrow \rho$
there exists $\uu \subseteq \kappa$, $|\uu| = \kappa$ such that
$\langle \alpha(\epsilon_1, \dots \epsilon_n) : \epsilon_1 < \dots < \epsilon_n ~\mbox{from $\uu$} \rangle$ is constant.
\end{fact}

Recall that $\de^+ = \{ X \subseteq I : X \neq \emptyset \mod \de \}$.

We will use an existence result from our paper \cite{MiSh:999}. 

\begin{rmk} For the purposes of this paper, the reader may simply
take $\de$ in Theorem \ref{excellent} to be $\lambda^+$-good; then Fact \ref{e:fact} holds by the Appendix to \cite{MiSh:999}. 
We will only use this quoted consequence of the definition, Fact \ref{e:fact}, summarized in Remark \ref{upgrade}.
\end{rmk}

\begin{thm-lit} \emph{(Excellent filters, Malliaris and Shelah \cite{MiSh:999})} \label{excellent}
Let $\lambda \geq \mu \geq \aleph_0$. Then there exists a $(\lambda, \mu)$-good triple $(I, \de, \gee)$ where
$|\gee| = 2^\lambda$, $\gee \subseteq {^I \mu}$, and
$\de$ is a regular, $\lambda^+$-excellent filter $\de$ on $\lambda$ $($thus also $\lambda^+$-good$)$.
\end{thm-lit}

\begin{fact} \emph{(\cite{MiSh:999} Claim 4.9)} \label{e:fact}
Let $\de$ be a regular, $\lambda^+$-excellent filter on $\lambda$. Then for any sequence
$\overline{A} = \langle A_u : u \in [\mu]^{<\aleph_0} \rangle \subseteq \de^+$ which is multiplicative $\mod \de$, meaning that
\[ \left( u, v \in [\mu]^{<\aleph_0} \right) \implies \left(  A_u \cap A_v = A_{u \cup v} \mod \de \right) \]
there is a sequence 
$\overline{A}^\prime = \langle A^\prime_u : u \in [\mu]^{<\aleph_0} \rangle \subseteq \de^+$ such that
\begin{enumerate}
\item $u \in [\mu]^{<\aleph_0} \implies A^\prime_u \subseteq A_u$ 
\item $u \in [\mu]^{<\aleph_0} \implies A^\prime_u = A_u \mod \de$
\item $\overline{A}^\prime$ is multiplicative, i.e. multiplicative modulo the trivial filter $\{ \lambda \}$.
\end{enumerate}
\end{fact}

\begin{rmk} \label{upgrade}
That is, when $\de$ is excellent we can upgrade ``multiplicative modulo $\de$'' to ``multiplicative.''
\end{rmk}

\begin{defn} \emph{(On cuts)}
\begin{enumerate}
\item Let $\de$ be a filter on $I$, $M$ a model, $\gee$ a family of independent functions and suppose
$(I, \de, \gee)$ is a pre-good triple.
Let $(\overline{a}_1, \overline{a}_2) =
(\langle a^1_\epsilon : \epsilon < \kappa_1 \rangle, \langle a^1_\epsilon : \epsilon < \kappa_2 \rangle)$ be
a pair of sequences of elements of $M^I$. Say that $(\overline{a}_1, \overline{a}_2)$ is a pre-cut $\mod D$ whenever:
\begin{enumerate}
\item $\langle a^1_\epsilon : \epsilon < \kappa_1 \rangle$ is increasing $\mod \de$
\item $\langle a^2_\zeta : \zeta < \kappa_2 \rangle$ is decreasing $\mod \de$
\item and $a^1_\epsilon < a^2_\zeta \mod \de$ for each $\epsilon < \kappa_1, \zeta < \kappa_2$
\end{enumerate}

\item Say that, moreover, $(\overline{a}_1, \overline{a}_2)$ is a pre-cut supported by $\gee^\prime \subseteq \gee$ $\mod D$
when it is a pre-cut $\mod \de$ and for some $\uu \subseteq \kappa$ unbounded in $\kappa$, and any $\epsilon < \zeta$ from $\uu$, the set
\[ A_{\epsilon, \zeta} = \{ s \in I : \exists x (a^1_\epsilon < a^1_\zeta < x < a^2_\zeta < a^2_\epsilon) \}/\de \]
is supported by $\gee^\prime$. That is, we can find a partition $\langle X_i : i < i(*) \rangle$ of $I$ whose members are elements of $\fin_s(\gee^\prime)$ 
and such that for each $i < i(*)$, either $X_i \subseteq A_{\epsilon, \zeta} \mod \de$ or $X_i \cap A_{\epsilon, \zeta} = \emptyset \mod \de$.

\item An ultrafilter $\de_*$ on $I$ \emph{has no $(\kappa_1, \kappa_2)$-cuts} when for some model $M$ of the theory of linear order
$M^I/\de$ does, where this means that for any $(\overline{a}_1, \overline{a}_2)$ which is a $(\kappa_1, \kappa_2)$-pre-cut $\mod \de_*$,
then $N$ realizes the type $\{ a^1_\epsilon < x < a^2_\zeta : \epsilon < \kappa_1, \zeta < \kappa_2 \}$. 
\end{enumerate}
\end{defn}

\begin{obs} \label{basis-obs}
Let $(I, \de, \gee)$ be a $(\lambda, \aleph_0)$-good triple and let $\langle B_{\epsilon, \zeta} : \epsilon < \zeta < \kappa \rangle$ be
a sequence of elements of $\de^+$, i.e. $\de$-nonzero subsets of $I$. Then for each $\epsilon < \zeta < \kappa$, there exist:
\begin{enumerate}
\item a maximal antichain $\langle A_{h_{\epsilon, \zeta, n}} /\de : n < \omega \rangle$
of $\mcp(I)/\de$ such that:
\begin{itemize}
\item each $h_{\epsilon, \zeta, n} \in \fin(\gee)$
\item $n < m < \omega \implies A_{h_{\epsilon, \zeta, n}} \cap A_{h_{\epsilon, \zeta, m}} = \emptyset \mod \de$ (hence is really empty, by the assumption
on the independence of $\gee$)
\end{itemize}
\item and truth values $\langle \trv_{\epsilon, \zeta, n} : n < \omega \rangle$ such that for each $n<\omega$
\[ A_{h_{\epsilon, \zeta, n}} \subseteq \left( B_{\epsilon, \zeta} \right)^{\trv_{\epsilon, \zeta, n}} \mod \de \]
\end{enumerate}
\noindent In other words, $B_{\epsilon, \zeta} = \bigcup \{ A_{h_{\epsilon, \zeta, n}} : n < \omega, \trv_{\epsilon, \zeta, n} = 1 \} \mod \de$. 
\end{obs}

\begin{proof}
Since $(I, \de, \gee)$ is a good triple, the elements of $\fin_s(\gee)$ are dense in $\mcp(I)/\de$. 
Since $(I, \de, \gee)$ is a $(\lambda, \aleph_0)$-good triple, the boolean algebra $\mcp(I)/\de$ has the $\aleph_1$-c.c.
Using these two facts, for any $\de$-nonzero $B_{\epsilon, \zeta}$, we may choose the sets and their exponents by induction on $n<\omega$. 
\end{proof}

Observation \ref{basis-obs} can be thought of as giving a pattern on which the set $B_{\epsilon, \zeta}$ is based, so we now 
give a definition of when two patterns are the same; we will then apply weak compactness to show we can extract a large set $\uu$ of $\kappa$
so that any two $\epsilon < \zeta$ from $U$ have the same associated pattern. 

\begin{defn} \emph{(Equivalent patterns)}
Let $(I, \de, \gee)$ be a $(\lambda, \aleph_0)$-good triple and let $\langle B_{\epsilon, \zeta} : \epsilon < \zeta < \kappa \rangle$ be
a sequence of elements of $\de^+$. For each $\epsilon < \zeta < \kappa$:
\begin{itemize}
\item let $\langle A_{h_{\epsilon, \zeta, n}} /\de : n < \omega \rangle$, $\langle \trv_{\epsilon, \zeta, n} : n < \omega \rangle$
be as in Observation \ref{basis-obs} 
\item let $\langle \gamma(i,\epsilon, \zeta) : i < i(\epsilon, \zeta) \rangle$ list 
$\bigcup \{ \dom(h_{\epsilon, \zeta, n} : n < \omega \rangle \}$ in increasing order
\end{itemize}

We define an equivalence relation $E = E(\langle B_{\epsilon, \zeta} : \epsilon < \zeta < \kappa \rangle)$ on 
$\{ (\epsilon, \zeta) : \epsilon < \zeta < \kappa \}$ by: $(\epsilon_1, \zeta_1) E (\epsilon_2, \zeta_2)$ iff:

\begin{itemize}
\item $i(\epsilon_1, \zeta_1) = i(\epsilon_2, \zeta_2)$
\item $\trv_{\epsilon_1, \zeta_1, n} = \trv_{\epsilon_2, \zeta_2, n}$ for each $n < \omega$
\item $\gamma(i, \epsilon_1, \zeta_1) \in \dom(h_{\epsilon_1, \zeta_1, n})$ iff
$\gamma(i, \epsilon_2, \zeta_2) \in \dom(h_{\epsilon_2, \zeta_2, n})$
\item if $\gamma(i, \epsilon_1, \zeta_1) \in \dom(h_{\epsilon_1, \zeta_1, n})$ then
$h_{\epsilon_1, \zeta_1, n}(\gamma(i, \epsilon_1, \zeta_1)) = h_{\epsilon_2, \zeta_2, n}(\gamma(i, \epsilon_2, \zeta_2))$
\end{itemize}
\end{defn}

\begin{rmk} \label{eq-size}
The equivalence relation just defined has at most $2^{\aleph_0}$ classes, since $i(\epsilon, \zeta) \leq \aleph_0$, 
the sequence $\overline{t}$ is countable with each member $\{0,1\}$-valued, and the sequence $\overline{h}$ is countable
with each member a function whose domain and range are each a finite set of natural numbers. 
\end{rmk}

\begin{claim} \label{claim-patterns}
Assume $\kappa$ is weakly compact.
Let $(I, \de, \gee)$ be a $(\lambda, \aleph_0)$-good triple and let $\langle B_{\epsilon, \zeta} : \epsilon < \zeta < \kappa \rangle$ be
a sequence of elements of $\de^+$. Then there is $\uu \subseteq \kappa$, $| \uu | = \kappa$ such that:

\begin{enumerate}
\item  all pairs from 
$\{ (\epsilon, \zeta) : \epsilon < \zeta < \kappa ~\mbox{are from $\uu$} \}$ are $E$-equivalent, and moreover
\item for any $i < j < i(\epsilon, \zeta)$, and any $\epsilon_1 < \zeta_1, \epsilon_2 < \zeta_2$ from $\uu$,
the truth or falsity of ``$\gamma(i, \epsilon_1, \zeta_1) = \gamma(j, \epsilon_2, \zeta_2)$'' depends only on the order type of
$\{ \epsilon_1, \epsilon_2, \zeta_1, \zeta_2 \}$. 
\end{enumerate}

Given $\uu$, we may thus speak of $i$, $\dom(h_{n})$, and $h_{n}(\gamma(j))$ as these depend only on the equivalence class.
Note that the actual identities of the functions $\{ \gamma(j, \epsilon, \zeta) : j < i\}$ do depend on $\epsilon, \zeta$, though by 
(2) their pattern of incidence does not.
\end{claim}

\begin{proof}
For the first clause, by Remark \ref{eq-size}, enumerate the $2^{\aleph_0}$-many possible $E$-equivalence classes
and define $f:[\kappa]^2 \rightarrow 2^{\aleph_0}$ by $(\epsilon, \zeta) \mapsto i$ when $(\epsilon, \zeta)$ is in the $i$th class.
Note that as $\kappa$ is weakly compact, therefore inaccessible, the range of this function is less than $\kappa$.
Thus by Fact \ref{wcc-fact}, we have a homogeneous subset $\uu_0 \subseteq \kappa$, $|\uu| = \kappa$. For the second clause, define
$f: [\uu_0]^4 \rightarrow {^{i \times i}\{0,1\}}$ and let $\uu \subseteq \uu_0$, $|\uu| = \kappa$ be a homogeneous subset.  
\end{proof}

The final use of weak compactness will be in the following observation, which will allow us to build the ultrafilter so that the 
cofinality of the construction is small while still addressing all $(\kappa, \kappa)$-cuts. 

\begin{obs} \label{wcc-lcf}
Recall that $\aleph_0 < \theta = \cf(\theta) < \kappa \leq \lambda$ and $\kappa$ is weakly compact.
Let $\delta$ be an ordinal with cofinality $\theta$. Suppose $\de_*$ is an ultrafilter on $I$, $|I| = \lambda$ and
$\de_*$ is constructed as the union of a continuous increasing chain of filters $\langle \de_\eta : \eta < \delta \rangle$,
for $\delta$ an ordinal with cofinality $\theta$. 
If $(\overline{a}_1, \overline{a}_2)$ is a $(\kappa, \kappa)$-cut for $\de_*$, then there is some $\rho < \delta$ such that
$(\overline{a}_1, \overline{a}_2)$ is already a $(\kappa, \kappa)$-cut for $\de_\rho$.

Moreover, if the chain begins with $\de_0, \gee$ where $(I, \de_0, \gee)$ is a $(\lambda, \aleph_0)$-good triple, then if we write 
$\gee = \bigcup \{ \gee_\alpha : \alpha < \delta \}$ as a union of increasing sets, then we may additionally conclude that
there is $\rho < \delta$ such that $(\overline{a}_1, \overline{a}_2)$ is a $(\kappa, \kappa)$-cut supported by $\gee_\rho \mod \de_\rho$.
[More precisely, for some $\uu \in [\kappa]^{\kappa}$, 
$(\overline{a}_1\rstr_\uu, \overline{a}_2\rstr_\uu)$ is a $(\kappa, \kappa)$-cut supported by $\gee_\rho \mod \de_\rho$.]
\end{obs}

\begin{proof}
For each $\epsilon < \zeta < \kappa$, let $A_{\epsilon, \zeta} = \{ s \in I : \exists x (a^1_\epsilon < a^1_\zeta < x < a^2_\zeta < a^2_\epsilon) \} \in \de_*$
be supported by $\gee_{\alpha(\epsilon, \zeta)}$ where $\alpha(\epsilon, \zeta) < \delta$. (See the proof of Observation \ref{basis-obs}, which applies to any
$\de_0$-positive set thus any element of $\de_*$, and note that the supporting sequence of functions obtained there is countable while 
$\theta = \cf(\theta) > \aleph_0$.) 
Let $C$ be a club of $\delta$ of order-type $\theta$, and without loss of generality
$\alpha(\epsilon, \zeta) \in C$. Since $\kappa$ is weakly compact and $\theta < \kappa$, there is $\uu \in [\kappa]^{\kappa}$ such that
$\alpha$ is constant on $\epsilon < \zeta$ from $\uu$. 
\end{proof}

We now turn to the construction of the ultrafilter. 

\begin{theorem} \label{wc-lcf}
Let $\kappa$ be weakly compact and let $\theta, \lambda$ be such that
$\aleph_0 < \theta = \cf{\theta} < \kappa \leq \lambda$. Then there is a regular ultrafilter $\de$ on $I$, $|I| = \lambda$
which has no $(\kappa, \kappa)$-cuts but $\lcf(\aleph_0, \de) = \theta$. 
\end{theorem}

\begin{proof}
We begin with  $\gee \subseteq {^I\mathbb{N}}$, $|G| = 2^\lambda$, and $\de$ so that $(I, \de_0, \gee)$ a $(\lambda, \aleph_0)$-good triple
and $\de_0$ is a $\lambda^+$-excellent filter, given by Theorem \ref{excellent} above. 
Let $\delta$ be an ordinal with cofinality $\theta$ divisible by $2^\lambda$ (we will use $2^\lambda \times \theta$).
Let $\langle g_\alpha : \alpha < \delta \rangle$ list $\gee$ with no repetition. Let $M$ be a model of the theory of linear order
with universe $\mathbb{N}$ (we will use $(\mathbb{N}, <)$). 

Let $\langle (\overline{a}^\alpha_1, \overline{a}^\alpha_2) : \alpha < \delta \rangle$ enumerate all pairs of $\kappa$-sequences of elements of 
$M^I$ (i.e. all potential $(\kappa, \kappa)$-pre-cuts) each occurring cofinally often.

We choose $\de_\alpha$ by induction on $1 \leq \alpha \leq \delta$ such that:
\begin{enumerate}
\item[(a)] $\de_\alpha$ is a filter on $I$ extending $\de_0$
\item[(b)] $\beta < \alpha \implies \de_\beta \subseteq \de_\alpha$ and the chain is continuous, i.e. $\alpha$ limit
implies $\de_\alpha = \bigcup \{ \de_\beta : \beta < \alpha \}$
\item[(c)] if $A \in \de_\alpha$ then $A/\de_0$ is supported by $\gee_\alpha := \{ g_\gamma : \gamma < 2^\lambda\alpha \}$
\item[(d)] If $A \subseteq I$ and $A/\de$ is supported by $\gee_\alpha$, then $A \in \de_\alpha$ or $I \setminus A \in \de_\alpha$
\item[(e)] If $(\overline{a}^\alpha_1, \overline{a}^\alpha_2)$ is a cut $\mod \de_\alpha$ hence is supported by $\gee_\alpha \mod \de_0$,
then for some $b_\alpha \in {^IM}$, for every $\epsilon < \kappa$ we have that the set 
\[ \{ s \in I : a^1_\epsilon[s] < b_\alpha[s] < a^2_\epsilon[s] \} \in \de_{\alpha+1}\]
hence any ultrafilter extending $\de_{\alpha+1}$ realizes this cut. 
\end{enumerate}

Before we justify this induction, note that it suffices: letting $\de = \de_\delta = \bigcup \{ \de_\alpha : \alpha < \delta \}$ be the resulting ultrafilter,
we have by Observation \ref{wcc-lcf} 
that $\de$ realizes all $(\kappa, \kappa)$-pre-cuts (since each of them will be addressed at some bounded stage), 
and we have that $\lcf(\aleph_0, \de) = \theta$ by Fact \ref{fact-lcf}.

For $\alpha = 0$, let $\de_\alpha = \de_0$.

For $\alpha$ limit, without loss of generality (see \S \ref{s:constr} above)
let $\de_\alpha$ be a maximal extension of $\bigcup \{ \de_\beta : \beta < \alpha \}$ satisfying 
(a), (b), (c); it will then necessarily satisfy (d). 

So let $\alpha = \beta + 1$. We consider the potential cut 
$(\overline{a}^\beta_1, \overline{a}^\beta_2)$, henceforth $(\overline{a}_1, \overline{a}_2)$. 
If for some $\epsilon, \zeta < \kappa$ the set 
$A_{\epsilon, \zeta} = \{ s \in I : \exists x (a^1_\epsilon < a^1_\zeta < x < a^2_\zeta < a^2_\epsilon) \} \notin \de_\beta$, 
then proceed as in the limit case. 

Otherwise, this stage has several steps.

\step{Step 1: Fixing uniform templates.}
Working $\mod \de_0$ (which recall is an excellent therefore good filter), each $A_{\epsilon, \zeta} = A^\beta_{\epsilon, \zeta}$ is supported by
$\gee_\beta \mod \de_0$, and is in $\de_0^+$. Thus applying Claim \ref{claim-patterns} to
$\langle A^\beta_{\epsilon, \zeta} : \epsilon < \zeta < \kappa \rangle$, let $\uu_\beta \in [\kappa]^{\kappa}$ be a homogeneous sequence
satisfying (1)-(2) of Claim \ref{claim-patterns} and let $i=i(\beta)$, $\dom(h^\beta_{n})$, 
$h^\beta_{n}(\gamma(i))$, and so forth be the associated data, which we may informally call patterns or ``templates''. For clarity, we write each
$\gamma(i,\epsilon, \zeta)$ [note: recall that these list the functions supporting $A_{\epsilon, \zeta}$ in increasing order]
as $\gamma(i,\epsilon, \zeta, \beta)$ to emphasize the stage in the construction.

Now, these ``templates'' are given for each of the sets $A_{\epsilon, \zeta}$, for \emph{pairs} $\epsilon < \zeta$ from $\uu$.
Stepping back for a moment, we would ultimately like to realize the given cut. We have in hand a distribution of the type
in which $\vp(x;a^1_\epsilon, a^2_\epsilon) = a^1_\epsilon < x < a^2_\epsilon$ is sent to 
$A_\epsilon = \{ s \in I : \exists x (a^1_\epsilon < x < a^2_\epsilon) \}$, and so we will eventually want to 
look for a multiplicative distribution in which each $A_\epsilon$ is refined to a smaller set.
Roughly speaking, our strategy will be to choose a refinement for $A_\epsilon$ by analyzing the limiting behavior of
the templates for $A_{\epsilon, \zeta}$ as $\zeta \rightarrow \kappa$. The hypothesis of weak compactness will give us enough
leverage to do this, by revealing strongly uniform behavior within any such sequence. Recall that in light of Claim \ref{claim-patterns},
what changes with $\zeta$ is the actual identity of the sequence of functions
$\langle \gamma(i, \epsilon, \zeta, \beta) : i < i_\beta \rangle \subseteq \gee_\beta$.  

\step{Step 2: ``Limit'' sequences from the templates.}
For each $\epsilon \in \uu$, we thus have a $\kappa$-sequence of countable sequences 
\[  \langle \langle \gamma(i, \epsilon, \zeta, \beta) : i < i_\beta \rangle : \zeta \in \uu \rangle \]
Looking ahead, we will want to choose the $\kappa$th row of this sequence, 
$\langle \gamma(i, \epsilon, \kappa, \beta) : i < i_\beta \rangle$ in a natural way. [The choice
will be justified in Steps 3-4.]
We first observe that Claim \ref{claim-patterns}(2) has strong consequences: 

\begin{enumerate}
\item[(i)] Fixing $\epsilon \in \uu$, 
 the sequence $\langle \gamma(i, \epsilon, \zeta, \beta) : \epsilon < \zeta ~\mbox{from $\uu$} \rangle$ is either constant or else 
any two of its elements are distinct, and
\item[(ii)] if for some $i_1 \neq i_2$, $\zeta_1 \neq \zeta_2$ we have
$\gamma(i_1, \epsilon, \zeta_1, \beta) = \gamma(i_2, \epsilon, \zeta_2, \beta)$ then both sequences
$\langle \gamma(i_1, \epsilon, \zeta, \beta) : \epsilon < \zeta ~\mbox{from $\uu$} \rangle$ and 
$\langle \gamma(i_2, \epsilon, \zeta, \beta) : \epsilon < \zeta ~\mbox{from $\uu$} \rangle$ are eventually equal to the same value and thus,
by (i), everywhere constant. 
\item[(iii)] For for $i_1, i_2, \epsilon_1, \epsilon_2$ we have that:
\newline there exist $\zeta_1, \zeta_2 > \epsilon_1, \epsilon_2$
such that 
\[ \gamma_1(i_1, \epsilon_1, \zeta_1, \beta) = \gamma_1(i_1, \epsilon_1, \zeta_2, \beta) =
\gamma_1(i_2, \epsilon_2, \zeta_1, \beta) = \gamma_1(i_2, \epsilon_2, \zeta_2, \beta) \]
if and only if  
\newline for all $\zeta > \epsilon_1, \epsilon_2$,  $\gamma(i_1, \epsilon_1, \zeta, \beta) = \gamma(i_2, \epsilon_2, \zeta, \beta)$.
\end{enumerate}

We may therefore choose ordinals
$\langle \gamma(i,\epsilon, \kappa, \beta) : \epsilon \in \uu, i < i(\beta) \rangle$, such that:

\begin{enumerate}
\item if $i<i(\beta), \epsilon \in \uu$ and $\gamma(i, \epsilon, \zeta_1, \beta) = \gamma(i, \epsilon, \zeta_2, \beta)$ whenever
$\zeta_2 > \zeta_1 > \epsilon$ are from $\uu$, then $\gamma(i, \epsilon, \kappa, \beta) = \gamma(i, \epsilon, \zeta, \beta)$ 
whenever $\zeta > \epsilon$ is from $\uu$
\item if $i < i(\beta)$, $\epsilon \in \uu_\beta$ and $\gamma(i, \epsilon, \kappa, \beta)$ is not defined by clause (1),
then $\gamma(i, \epsilon, \kappa, \beta) \in [2^\lambda\beta, 2^\lambda\alpha)$, i.e. in $\gee_\alpha \setminus \gee_\beta$, subject to:
\item if $\epsilon_1, \epsilon_2 \in \uu$ and $i_1, i_2 < i(\beta)$ then $\gamma(i_1, \epsilon_1, \kappa, \beta) = \gamma(i_2, \epsilon_2, \kappa, \beta)$ 
iff $\gamma(i_1, \epsilon_1, \zeta, \beta) = \gamma(i_2, \epsilon_2, \zeta, \beta)$ for all $\zeta > \epsilon_1, \epsilon_2$ from $\uu_\beta$.
\end{enumerate}

For each $\epsilon \in \uu$, let $h_{n, \epsilon, \kappa, \beta}$ be defined in the obvious way, i.e. as the template functions considered over the countable sequence
$\{ \gamma(i, \epsilon, \kappa, \beta) : i < i_\beta \}$. Then set
\[ A^\beta_{\epsilon, \kappa} = \bigcup \{ A_{h_{n, \epsilon, \kappa, \beta}} : \trv_{n, \beta} = 1 \} \]
\noindent For each finite $u \subseteq \uu$ let
\[ B_u = \{ t \in I : M \models (\exists x) (\bigwedge_{\epsilon \in u} a^\epsilon_1[t] < x < a^\epsilon_2[t] \} \]
and let 
\[ B^\prime_u = \bigcap \{ A^\beta_{\epsilon,\kappa} : \epsilon \in u \} \]
With these definitions in hand, we turn to Step 3.

\step{Step 3: $B_u \supseteq B^\prime_u \mod \de_0$.}

In this step we verify that for each $u \in [\uu]^{<\aleph_0}$, $B_u \supseteq B^\prime_u \mod \de_0$. Why? Simply 
because of the uniformity of the templates, the choice at $\kappa$ and the independence of $\gee$. Informally,
the sequence of sequences in Step 2 was $\emptyset$-indiscernible in the Boolean algebra $\mcp(I)/\de_0$,
and the generic choice in Step 2 added a $\kappa$th element (sequence) to this sequence. 

More specifically, it suffices that for $u \in [\uu]^{<\aleph_0}$,
\[(**)~~~~ \left(B^\prime_u := \bigcap_{\epsilon \in u} A^\beta_{\epsilon,\kappa} \subseteq B_u \mod \de_0 \right)
\iff \left( \bigcap_{\epsilon \in u} A^\beta_{\epsilon,\zeta} \subseteq B_u \mod \de_0 ~\mbox{for sufficiently large $\zeta$} \right) \]

Why does $(**)$ suffice for Step 3? Because the right-hand side will always hold:
fix $u \in [\uu]^{<\aleph_0}$, let $\zeta > \max u$ and recall the definition of
$A^\beta_{\epsilon, \zeta}$ from the beginning of the inductive proof. 

\noindent\underline{Details: the case $|u| = 1$.} Here we justify why ``continuing the indiscernible sequence'' retains the
right relationship to $B_u$; all key ideas of the proof appear in this notationally simplest case.
Suppose that $u = \{ \epsilon \}$ for some $\epsilon \in \uu$. 
By definition of $A^\beta_{\epsilon, \kappa}$, it will suffice to show that $\trv_{n, \beta} = 1 \implies
A_{h_{n, \epsilon, \kappa, \beta}}$. Fix some such $n$.

Let us recall the picture. We have a sequence of countable sequences, 
\[  \langle \langle \gamma(i, \epsilon, \zeta, \beta) : i < i_\beta \rangle : \zeta \in \uu \cup \{ \kappa \} \rangle \]
or if the reader prefers, an $i_\beta \times \kappa$-array, whose elements are functions from $\gee$. We know from Step 2
that this array is strongly uniform in various ways: for instance, the elements in each column $\{ \gamma(i, \epsilon, \zeta, \beta) : \zeta \in \uu \cup \{ \kappa \} \}$ are either pairwise equal or pairwise distinct.

The function $h_{n,\epsilon, \kappa, \beta} \in \fin(\gee)$ under consideration assigns finitely many elements of row $\kappa$, say the elements 
with indices $\{ i_1, \dots i_\ell \} \subseteq i_\beta$ [which are, themselves, functions belonging to $\gee$] specific integer values $\{ m_1, \dots m_\ell \} \subseteq \mathbb{N}$. 
Since the domain of $h_{n, \epsilon, \kappa, \beta}$ is $\gee$, rather than $i_\beta$, it will simplify notation to define a row function
$\bz: \uu \cup \{ \kappa \} \rightarrow \fin_s(\gee)$ which, to each row $\zeta$ in the array, assigns the set
$A_{h_{n,\epsilon, \zeta, \beta}}$. 

We make a series of observations.

(1) By construction, for each $\zeta \in \uu$, $\bz(\zeta) \subseteq B_{\{ \epsilon \}} \mod \de_0$. 

Suppose that (1) fails for $\bz(\kappa)$, that is, 
\[ X := \{ s \in I : s \in \bz(\kappa), s \notin B_{ \{\epsilon\} } \} \neq \emptyset \mod \de_0 \]
As $(I, \de_0, \gee)$ is a good triple, $\fin_s(\gee)$ is dense in $(\de_0)^+$, and there is some $Y \in \fin_s(\gee), Y \subseteq X \mod \de_0$. 
Let $h_Y \in \fin(\gee)$ be such that $A_{h_Y} = Y$. 

(2) Without loss of generality, since $Y \subseteq \bz(\kappa) = A_{h_{n,\epsilon, \kappa, \beta}}$,
we may assume that any functions in the domain of ${h_{n,\epsilon, \kappa, \beta}}$ are in the domain of $h_Y$ (and that on their common domain,
${h_{n,\epsilon, \kappa, \beta}}$ and $h_Y$ agree). There are two cases.

Case 1: There is some $\zeta \in \uu$ such that $\bz(\zeta) \cap Y \neq \emptyset \mod \de_0$. This contradicts (1).

Case 2: Not case 1, that is, for every $\zeta \in \uu$, $\bz(\zeta) \cap Y = \emptyset \mod \de_0$. Notice that:

(3) As $\{ \bz(\zeta) : \zeta \in \uu \} \cup \{ Y \} \subseteq \fin_s(\gee)$, the only way this can happen is if there is an explicit contradiction
in the corresponding functions, i.e  if for each $\zeta \in \uu$ there is $f \in \dom({h_{n,\epsilon, \zeta, \beta}}) \cap 
\dom(h_Y)$ and ${h_{n,\epsilon, \zeta, \beta}}(f) \neq h_Y(f)$. 

The functions in $\dom(h_Y)$ have one of three sources:
\begin{itemize}
\item[(i)] elements of $\gee_\beta$ which already belong to $\dom({h_{n,\epsilon, \kappa, \beta}})$, and thus also to each
$\dom({h_{n,\epsilon, \zeta, \beta}})$, by the construction in Step 2
\item[(ii)] elements of $\gee_\beta$ which do not belong to $\dom({h_{n,\epsilon, \kappa, \beta}})$, and thus to no
more than one $\dom({h_{n,\epsilon, \zeta, \beta}})$ by the construction in Step 2
\item[(iii)] elements of $\gee \setminus \gee_\beta$
\end{itemize}

By the choice of row $\kappa$ in Step 2, for all $\zeta \in \uu$, 
${h_{n,\epsilon, \kappa, \beta}}$ and ${h_{n,\epsilon, \zeta, \beta}}$ agree on functions from $\gee_\beta$.
Since $h_Y$ is compatible with ${h_{n,\epsilon, \kappa, \beta}}$ by construction (2), there can be no incompatibility with 
${h_{n,\epsilon, \zeta, \beta}}$ on
functions of type (i), and since $\dom({h_{n,\epsilon, \zeta, \beta}}) \subseteq \gee_\beta$ by inductive hypothesis, 
neither will there be an incompatibility with (iii). As noted, each conflict of type (ii) rules out at most one $\zeta \in \uu$.
Since $\dom(h_Y)$ is finite, it must be that the instructions ${h_{n,\epsilon, \zeta, \beta}} \cup h_Y$ are compatible 
for all but finitely many $\zeta$. Thus for all but finitely many $\zeta \in \uu$, $\bz(\zeta) \cap Y \neq \emptyset$ by (3).

We have shown that if $\bz(\kappa) \not\subseteq B_{ \{ \epsilon \} }$, then for some (in fact, nearly all) $\zeta \in \uu$, 
$\bz(\zeta) \not\subseteq B_{ \{ \epsilon \} } \mod \de_0$, contradicting (1). 

Conversely, if for more than finitely many $\zeta \in \uu$, $\bz(\zeta) \not\subseteq B_{ \{ \epsilon \} } \mod \de_0$
this will be reflected in the functions of type (i), so inherited by $\bz(\kappa)$. This completes the proof.

\noindent\underline{The case $|u| > 1$.}
The argument for $|u| > 1$ involves more notation, but no new ideas, since condition (3) of Step 2 guarantees that we
have the same level of uniformity across finitely many $\epsilon$-sequences.

\step{Step 4: A multiplicative refinement.}
In this step we tie up loose ends, realize the type and finish the inductive definition of $\de_\alpha$.
We make a sequence of assertions.

First, the sequence $\langle B^\prime_u : u \in [\uu]^{<\aleph_0} \rangle \subseteq (\de_0)^+$ is multiplicative.

Second, for each $u \in [\uu]^{<\aleph_0}$, $B_u \supseteq B^\prime_u \mod \de_0$, by Step 3.

Third, by definition of the sets $B^\prime_u$, recalling the fact that $(I, \de_0, \gee)$ is 
a good triple, we have that $\de_\beta \cup \{ B^{\prime}_u : u \in [\uu]^{<\aleph_0} \}$ generates a filter.

However, this is not yet enough: a multiplicative refinement must be genuinely contained inside the original sequence,
rather than simply contained $\mod \de_0$. So let us define a third sequence
\[ \langle B^{\prime\prime}_u : u \in [\uu]^{<\aleph_0} \rangle, ~\mbox{where}~ 
u \in [\uu]^{<\aleph_0} \implies B^{\prime\prime}_u  = B_u \cap B^{\prime}_u \]

We have transferred the problem: now, on one hand
(1) $B^{\prime\prime}_u \subseteq B_u$ (without ``mod $\de$''), while on the other hand (2) 
$\langle B^{\prime\prime}_u : u \in [\uu]^{<\aleph_0} \rangle$ is multiplicative $\mod \de_0$.

As we chose $\de_0$ to be $\lambda^+$-excellent, by Fact \ref{e:fact}, there is 
$\langle B^{*}_u : u \in [\uu]^{<\aleph_0} \rangle \subseteq \de^+$
refining $\langle B^{\prime\prime}_u : u \in [\uu]^{<\aleph_0} \rangle$, which is indeed multiplicative,
and which satisfies $B^*_u = B^{\prime\prime}_u = B^\prime_u \mod \de$.
A fortiori this fourth sequence refines the original distribution, $\langle B_u : u \in [\uu]^{<\aleph_0} \rangle$. 

Now we finish. Let $\de^\prime_\alpha$ be the filter generated by $\de_\beta \cup \{ B^{*}_u : u \in [\uu]^{<\aleph_0} \}$.
Since for each $u \in [\uu]^{<\aleph_0}$, $B^*_u = B^\prime_u \mod \de_0$, this is a (nontrivial) filter.
Any ultrafilter extending $\de^\prime_\alpha$ will realize the cut 
$(\overline{a}_1, \overline{a}_2) = (\overline{a}^\beta_1, \overline{a}^\beta_2)$, since its distribution has a multiplicative
refinement.
(By transitivity of linear order and the fact that $\uu$ is cofinal in $\kappa$,
there is no loss in realizing the cut restricted to $\uu$.)

Finally, as before, let $\de_\alpha$ be a maximal extension of $\de^\prime_\alpha$ satisfying 
(a), (b), (c); it will then necessarily satisfy (d). This completes the inductive step, and thus the proof.
\end{proof}

\subsection{Discussion}
In the remainder of this section, we discuss some variants of Theorem \ref{wc-lcf}. 

\begin{rmk} \label{wc-use}
The hypothesis ``$\kappa$ is weakly compact'' in Theorem \ref{wc-lcf} was used in two key places:
\begin{enumerate}
\item to extract a very uniform subsequence,  Claim \ref{claim-patterns}
\item to ensure that each cut was supported at some bounded stage in the construction, Observation \ref{wcc-lcf}
\end{enumerate}
\end{rmk}

We may wish to avoid large cardinal hypotheses, which we can do using the following polarized partition relation
(of course, the result will no longer be about symmetric cuts).

\begin{defn} \cite{e-h-m-r}
The \emph{polarized partition relation}
\[ \left( {\begin{array}{c}
					\kappa_1 \\
					\kappa_2 \\
					\end{array}}   \right) \rightarrow \left( {	\begin{array}{c}
																											\kappa_1 \\
																											\kappa_2 \\
																											\end{array}} \right)^{(1,1)}_{2^{\aleph_0}} \]								
holds when for every coloring of $\{ ( \alpha, \beta ) : \alpha < \kappa_1, \beta < \kappa_2 \}$
by at most $2^{\aleph_0}$-many colors, there exist $X\in [\kappa_1]^{\kappa_1}, Y \in [\kappa_2]^{\kappa_2}$
such that $\{ ( \alpha, \beta ) : \alpha \in X, \beta \in Y \}$ is monochromatic.
\end{defn}

\begin{fact} \label{ppr-fact} \cite{e-h-m-r} 
Suppose $\kappa_1, \kappa_2$ are regular and that $2^{\aleph_0} < {2^{\kappa_1}} < \kappa_2$. Then
\[ \left( {\begin{array}{c}
					\kappa_1 \\
					\kappa_2 \\
					\end{array}}   \right) \rightarrow \left( {	\begin{array}{c}
																											\kappa_1 \\
																											\kappa_2 \\
																											\end{array}} \right)^{(1,1)}_{2^{\aleph_0}} \]								
\end{fact}

\begin{concl}
Suppose $\kappa_1, \kappa_2 \leq \lambda$ satisfy the hypotheses of Fact \ref{ppr-fact} and that $\theta = \cf(\theta) \leq 2^{\aleph_0}$. 
Then there is a regular ultrafilter $\de$ on $I$, $|I| = \lambda$
which has no $(\kappa_1, \kappa_2)$-cuts but $\lcf(\aleph_0, \de) = \theta$. 
\end{concl}

\begin{proof}
Suppose we want to ensure realization of all $(\kappa_1, \kappa_2)$-cuts. We proceed
just as in the proof of Theorem \ref{wc-lcf}, with the following two changes corresponding to the two parts of Remark 
\ref{wc-use}. The polarized partition relation will allow us to 
extract a cofinal sub-cut in accordance with condition (1), as by Remark \ref{eq-size} there are only continuum many equivalence 
classes. It will likewise allow us to carry out the argument of Observation \ref{wcc-lcf} whenever the given $\theta \leq 2^{\aleph_0}$;
note that without the stronger assumption of weak compactness, it is no longer sufficient to assume $\theta < \kappa$. 
\end{proof}

\br
In a paper in preparation \cite{MiSh:998}, we investigate further the set of possible cofinalities of cuts in ultrapowers of linear order.


\begin{thebibliography}{50}
\bibitem{ck73} C. C. Chang and H. J. Keisler, \emph{Model Theory}, North-Holland, 1973. 

\bibitem{dow} A. Dow, ``Good and OK ultrafilters.'' 
Trans. AMS, Vol. 290, No. 1 (July 1985), pp. 145-160.

\bibitem{DzSh} M. D\v{z}amonja and S. Shelah, ``On $\vartriangleleft\sp*$-maximality,'' Ann. Pure Appl. Logic 125 (2004) 119--158

\bibitem{e-h-m-r} E. Erd\"{o}s, A. Hajnal, A. Mate, and R. Rado. \emph{Combinatorial set theory: partition relations for cardinals.}
Studies in Logic and the Foundations of Mathematics 106. North-Holland, 1984.

\bibitem{ek} R. Engelking and M. Kar\l owicz, ``Some theorems of set theory and their topological consequences.''
Fund. Math. 57 (1965) 275--285. 

\bibitem{kanamori} A. Kanamori, \emph{The Higher Infinite}. Springer, 2003. 

\bibitem{keisler-1} H. J. Keisler, ``Good ideals in fields of sets,'' Ann. of Math. (2) 79 (1964), 338-359.

\bibitem{keisler} H. J. Keisler, ``Ultraproducts which are not saturated.'' J. Symbolic Logic 32 (1967) 23--46.


\bibitem{KSV} J. Kennedy, S. Shelah, and J. Vaananen. ``Regular Ultrafilters and Finite Square Principles.'' J. Symbolic Logic 73 (2008) 817-823.

\bibitem{kunen} K. Kunen, ``Ultrafilters and independent sets.'' Trans. Amer. Math. Soc. 172 (1972), 299--306. 

\bibitem{mm1} M. Malliaris, ``Realization of $\vp$-types and Keisler's order.'' Ann. Pure Appl. Logic 157 (2009), no. 2-3, 220--224

\bibitem{mm-thesis} M. Malliaris, Ph. D. thesis, U.C. Berkeley (2009). Available at http://math.uchicago.edu/$\sim$mem.
\bibitem{mm4} M. Malliaris, ``Hypergraph sequences as a tool for saturation of ultrapowers.'' J. Symbolic Logic, 77, 1 (2012) 195-223.
%
\bibitem{mm5} M. Malliaris, ``Independence, order and the interaction of ultrafilters and theories.'' (2010) To appear, Ann. Pure Appl. Logic. 
http://dx.doi.org/10.1016/j.apal.2011.12.010.

\bibitem{MiSh:996} M. Malliaris and S. Shelah, ``Constructing regular ultrafilters from a model-theoretic point of view.'' 
http://arxiv.org/abs/1204.1481.

\bibitem{MiSh:998} M. Malliaris and S. Shelah, ``Cofinality spectrum theorems in model theory, set theory and general topology.''  

\bibitem{MiSh:999} M. Malliaris and S. Shelah, ``A dividing line within simple unstable theories.'' 

\bibitem{Sh:a} 
S. Shelah, \emph{Classification Theory}, North-Holland, 1978. 


\bibitem{Sh:c} 
S. Shelah, \emph{Classification Theory and the number of non-isomorphic models}, rev. ed., North-Holland, 1990. 

\bibitem{Sh500} �
S. Shelah, ``Toward classifying unstable theories.'' Annals of Pure and Applied Logic 80 (1996) 229--255. 

\bibitem{ShUs}
S. Shelah and A. Usvyatsov, ``More on ${\rm SOP}_1$ and ${\rm SOP}_2$.'' Ann. Pure Appl. Logic 155 (2008), no. 1, 16--31. 
\end{thebibliography}
\end{document}